\documentclass[11pt] {article}
\usepackage{amsmath,amssymb,amsfonts}
\usepackage[utf8]{inputenc} 
\usepackage{amsthm}

\usepackage{tikz}
\usepackage{hyperref}
\usepackage{graphicx}
\usepackage{amsfonts}

\usepackage{color}
\usepackage[english]{babel}
\usepackage{ifthen}
\usepackage{graphicx}
\usepackage{tikz}
\usetikzlibrary{decorations.markings}
\usetikzlibrary{plotmarks}
\usetikzlibrary{patterns}
\usepackage{pgfplots}
\usepgfplotslibrary{fillbetween}
\usepackage{makeidx}
\usetikzlibrary{calc}

\usepackage{geometry} 
\usepackage{graphicx} 

\usepackage{array} 
\usepackage{paralist} 
\usepackage{verbatim} 
\usepackage{subfig} 

\usepackage{textpos}

\newcommand{\R}{\mathbb{R}}

\newcommand{\C}{\mathbb{C}}

\newcommand{\D}{\mathbb{D}}

\newcommand{\n}{\vec{n}}
\newcommand{\s}{\mathbb{S}}
\newcommand{\Ar}{\mathring{A}}
\newcommand{\zb}{\bar{z}}

\newcommand{\Lr}{\vec{L}}
\tikzset { domaine/.style 2 args={domain=#1:#2} }

\newtheorem{theo}{Theorem}[section]
\newtheorem*{theo*}{Theorem}

\newtheorem{prop}[theo]{Proposition}
\newtheorem*{prop*}{Proposition}
\newtheorem{lem}{Lemma}[section]
\newtheorem{cor}{Corollary}[section]
\newtheorem*{cor*}{Corollary}
\newtheorem{de}{Definition }[section]

\newtheorem{remark}{Remark}[section]

\newcommand{\nocontentsline}[3]{}
\newcommand{\tocless}[2]{\bgroup\let\addcontentsline=\nocontentsline#1{#2}\egroup}

\title{An $\varepsilon$-regularity result with mean curvature control for Willmore immersions and application to minimal bubbling.}
\author{Nicolas Marque \thanks{Institut Mathématique de Jussieu, Paris VII, Bâtiment Sophie Germain, Case 7052, 75205 Paris Cedex 13, France. E-mail address: nicolas.marque@imj-prg.fr}}
\date{\today} 

\begin{document}

\maketitle

\abstract{ In this paper we prove a convergence result for sequences of Willmore immersions with simple minimal bubbles. To this end we replace the total curvature control in T. Rivière's proof of the $\varepsilon$-regularity for Willmore immersions by a control of the local Willmore energy.}
\maketitle

\tableofcontents
\hspace{0.5cm}
\section{Introduction}
The following is primarily concerned with the study of Willmore immersions in $\R^3$. Let $\Phi$ be an immersion from a closed Riemann surface $\Sigma$ into $\R^3$. We denote by $g:= \Phi^* \xi$ the pullback by $\Phi$ of the euclidean metric $\xi$ of $\R^3$, also called the first fundamental form of $\Phi$ or the induced metric. Let $d\mathrm{vol}_g$ be the volume form associated with $g$. The Gauss map of $\Phi$ is the normal to the surface. In local coordinates $(x,y)$: 
$$ \n:= \frac{\Phi_x \times \Phi_y}{\left|\Phi_x \times \Phi_y \right|},$$
where $ \Phi_x = \partial_x \Phi$, $\Phi_y = \partial_y \Phi$ and $\times$ is the usual vectorial product in $\R^3$. 
Denoting  $\pi_{\n}$ the orthonormal projection on the normal (meaning $\pi_{\n}(v) = \langle \n , v \rangle \n$), the second fundamental form of $\Phi$ at the point $p \in \Sigma$ is defined as follows.
$$ \vec{A}_p (X,Y):= A_p(X,Y) \n:=\pi_{\n} \left( d^2 \Phi \left(X,Y \right) \right) \text{ for all } X,Y \in T_p\Sigma.$$
The mean curvature of the immersion at $p$ is then 
$$ \vec{H}(p)= H(p) \n= \frac{1}{2} \mathrm{Tr}_g \left( A \right) \n,$$
while its tracefree second fundamental form is 
$$\Ar_p (X,Y) = A_p(X,Y)  - H(p) g_p(X,Y).$$
The Willmore energy is  defined as $$W(\Phi):= \int_\Sigma H^2 d\mathrm{vol}_g.$$ Willmore immersions are critical points of this Willmore energy.
The Willmore energy was already  under scrutiny in the XIXth century in the study of elastic plates, but to our knowledge W. Blaschke was the first to state (see \cite{MR0076373}) its invariance by conformal diffeomorphisms of $\R^3$  (which was later rediscovered by T. Willmore, see  \cite{bibwill}) and to study it in the context of conformal geometry. 

While the Willmore energy is the canonically studied Lagrangian, its invariance is contextual. Indeed, $W$ is not invariant by inversions whose center is on the surface (the simplest example is the euclidean sphere which is sent to a plane once inverted at one of its points). The true \emph{pointwise} conformal invariant (as shown by T. Willmore, \cite{bibwill}) is in fact $\big| \Ar_p \big| d\mathrm{vol}_{g_p}$. The tracefree curvature and the total curvature are then two relevant energies, respectively defined as follows: 
$$\begin{aligned}E( \Phi ) &:= \int_\Sigma \big| {A} \big|^2_g d\mathrm{vol}_g = \int_\Sigma \left|\nabla_g \n \right|^2 d\mathrm{vol}_g,\\
 \mathcal{E} ( \Phi ) &:= \int_\Sigma \big| {\Ar} \big|^2_g d\mathrm{vol}_g. \end{aligned}$$
Quick and straightforward computations (done in appendix  \ref{formulasconformalimmersion} in a conformal chart) ensure that both
\begin{equation} \label{courbtot} E ( \Phi )  = 4 W(\Phi) - 4 \pi \chi(\Sigma) \end{equation}
with $\chi(\Sigma)$ the Euler characteristic of $\Sigma$, and 
\begin{equation} \label{courbsanstrace} \mathcal{E}(\Phi) = 2W(\Phi) - 4\pi \chi(\Sigma).\end{equation}
The invariance of $W$ when the topology of the surface is not changed then follows from (\ref{courbsanstrace}). 
  A Willmore surface is thus a critical point of $W$, $E$ and $\mathcal{E}$.

Great strides in the understanding of Willmore surfaces were made by E. Kuwert and R. Schatzle (see  \cite{MR2119722}  and  \cite{bibkuwertschatzlewillmoreflow}) and then by T. Rivière, who introduced the  framework of weak immersions\footnote{Denoted $\mathcal{E}(\Sigma)$, see definition \ref{weakimmersions} for more details} and  Willmore conservation laws (see for instance  theorem I.4 in \cite{bibanalysisaspects}). Y. Bernard later showed in \cite{bibnoetherwill} that  they stemmed from the conformal invariance of $W$. These conservation laws allow for the introduction on simply connected domains of auxiliary Willmore quantities $\Lr$, $S$ and $\vec{R}$, defined as follows 
\begin{equation}\label{LRSintro}\begin{aligned}&  \nabla^\perp \Lr = \nabla \vec{H} -3 \pi_{\n} \left( \nabla \vec{H} \right) + \nabla^\perp \n \times \vec{H}, \\  &\nabla^\perp S = \langle \Lr, \nabla^\perp \Phi \rangle, \\
& \nabla^\perp \vec{R} = \Lr \times \nabla^\perp \Phi + 2H \nabla^\perp \Phi. \end{aligned}\end{equation}
The second and third quantities, $S$ and $\vec{R}$, are remarkable in that they solve a Jacobian-like system that allows the use of Wente's lemmas 
\begin{equation}
\label{systemenRSPhiannexi}
\left\{
\begin{aligned}
\Delta S &= - \left\langle \nabla \n, \nabla^\perp \vec{R} \right\rangle \\
\Delta \vec{R} &= \nabla \n \times \nabla^\perp \vec{R} + \nabla^\perp S \nabla \n\\
\Delta \Phi &= \frac{1}{2} \left( \nabla^\perp S. \nabla \Phi + \nabla^\perp \vec{R} \times \nabla \Phi \right).
\end{aligned}
\right.
\end{equation}
 Exploiting these quantities and system (\ref{systemenRSPhiannexi}) yielded a variety of $\varepsilon$-regularity results for Willmore immersions (following is a combination of theorem I.5 in \cite{bibanalysisaspects} and theorem I.1 in \cite{bibenergyquant}).

\begin{theo}
\label{epsilonregularityRivierecontrolphi}
Let $\Phi \in \mathcal{E}\left(\D \right)$
 be a conformal weak Willmore immersion.  Let  $\n$ denote its Gauss map, $H$ its mean curvature and $\lambda = \frac{1}{2}\log \left( \frac{\left| \nabla \Phi \right|^2}{2} \right)$ its conformal factor.
We assume $$ \left\| \nabla \lambda \right\|_{L^{2,\infty} \left( \D \right)} \le C_0.$$
Then there exists $\varepsilon_0>0$ such that if
\begin{equation} \label{petitenenergydanslepsilonregularite} \int_{\D} \left| \nabla \n \right|^2  < \varepsilon_0, \end{equation}
then for any $r<1$ 
and for any $k \in \mathbb{N}$ 
$$\begin{aligned}
&\| \nabla^k \n \|^2_{L^\infty \left( \D_{r} \right)} \le C \int_\D \left| \nabla \n \right|^2, \\
&\| e^{-\lambda} \nabla^k \Phi \|^2_{L^\infty \left( \D_r \right)} \le   C  \left( \int_\D \left| \nabla \n \right|^2 +1 \right),
\end{aligned}$$
with $C$ a real constant depending on $r$, $C_0$ and $k$.
\end{theo}
This theorem in fact followed a preexisting result of E. Kuwert and R. Schätzle (see \cite{bibkuwschat}). T. Rivière introduced the auxiliary quantities, pinpointed their key role and originally wrote a proof in arbitrary codimension (see for instance theorem I.5 in \cite{bibanalysisaspects}). 

Such results induce a now classical concentration/compactness dialectic, as originally developed by J. Sacks and K. Uhlenbeck, for Willmore surfaces with bounded total curvature (or alternatively, given (\ref{courbtot}), bounded Willmore energy and topology). In essence, sequences of Willmore surfaces converge smoothly away from concentration points, on which trees of Willmore spheres are blown (see \cite{bubbles} for an exploration of the bubble tree phenomenon in another simpler case). Y. Bernard and T. Rivière developed an energy quantization result for such sequences of Willmore immersions assuming their conformal class is in a compact of the Teichmuller space (see theorem I.2 in \cite{bibenergyquant}). P. Laurain and T. Rivière then showed one could replace  the bounded conformal class hypothesis by a weaker convergence of residues linked with the conservation laws. Since we will work with bounded conformal classes, we here give abridged versions of theorems I.2 and I.3 of \cite{bibenergyquant}.

\begin{theo}
\label{energyquandberriv}
Let $\Phi_k$ be a sequence of Willmore immersions of a closed surface $\Sigma$. Assume that 
$$\limsup_{k\rightarrow \infty} W(\Phi_k) < \infty,$$
and that the conformal class of $\Phi^*_k \xi$ remains within a compact subdomain of the moduli space of $\Sigma$. Then, modulo extraction of a subsequence, the following energy identity holds 
$$ \lim_{k\rightarrow \infty} W(\Phi_k) = W (\Phi_\infty) + \sum_{s=1}^p W(\eta_s) + \sum_{t=1}^q \left[ W( \zeta_t) - 4 \pi \theta_t \right],$$
where $\Phi_\infty$ (respectively $\eta_s$, $\zeta_t$) is a possibly branched smooth immersion of $\Sigma$  (respectively $\s^2$) and $\theta_t \in \mathbb{N}$.
Further, there exists $a^1 \dots a^n \in \Sigma$ such that  $$ \Phi_k \rightarrow \Phi_\infty \text{ in } C^\infty_{\mathrm{loc}} \left( \Sigma \backslash \{a^1,\dots, a^n \} \right) $$ up to conformal diffeomorphisms of $\R^3\cup \{ \infty\}$.
Moreover, there exists a sequence of radii $\rho^s_k$, points $x^s_k \in \C$ converging to one of the $a^i$
 such that up to conformal diffeomorphisms of $\R^3$
$$\Phi_k\left( \rho^s_k y + x^s_k \right) \rightarrow \eta_s \circ \pi^{-1} (y) \text{ in } C^\infty_{\mathrm{loc}} \left( \C \backslash \{ \text{finite set} \} \right).$$
 Finally,
 there exists a sequence of radii $\rho^t_k$, points $x^t_k \in \C$ converging to one of the $a^i$
such that up to conformal diffeomorphisms of $\R^3$
$$\Phi_k\left( \rho^t_k y + x^t_k \right) \rightarrow \iota_{p_t} \circ \zeta_t \circ \pi^{-1} (y) \text{ in } C^\infty_{\mathrm{loc}} \left( \C \backslash \{ \text{finite set} \} \right).$$ Here, $\iota_{p_t}$ is an inversion at $p \in \zeta_t ( \s^2)$. The integer $\theta_t$ is the density of $\zeta_t$ at $p_t$.
\end{theo}
While theorem \ref{energyquandberriv} states an energy quantization for $W$, equality VIII.8 in \cite{bibenergyquant} offers in fact a stronger energy quantization for $E$. The $a^i$ are the aforementioned concentration points and the $\eta_s$ and $\iota_{p_t} \circ \zeta_t$ are the bubbles blown on those concentration points. More precisely, the $\eta_s$ are the compact bubbles, and the $\iota_{p_t} \circ \zeta_t$ the non-compact ones.  Non-compact bubbles stand out as a consequence of the conformal invariance of the problem (see \cite{MR2876249} to compare with the bubble tree extraction in the constant mean curvature framework).  One might notice that $W( \iota_{p_t} \circ \zeta_t)= W( \zeta_t) - 4 \pi \theta_t$, and deduce that if $W( \zeta_t) = 4 \pi \theta_t$, then the bubble $\iota_{p_t} \circ \zeta_t$ is minimal. This case, which we will refer to as \emph{minimal bubbling} will be of special interest to us in this article. Further, if there is only one bubble at a given concentration point  we will call the bubbling \emph{simple}. 
 Figures \ref{scpm} and \ref{scpmbis} illustrate two a priori possible bubbling configurations: figure \ref{scpm} presents a simple minimal bubble, while figure \ref{scpmbis} displays a bubble tree.

These bubble  trees have been studied with success: works  from Y. Li in \cite{MR3511481} proved that bubble trees cannot be embedded, and P. Laurain and T. Rivière  (see  \cite{MR3843372}) ensured that compact simple bubbles cannot appear. Non-compact bubbling thus remains the only simple bubbling to consider, with minimal simple bubbling being a prominent example (with the lowest total energy) and the main subject of the present paper.  As an illustration, one can keep in mind the configuration of figure \ref{scpm}: an Enneper surface (the bubble) is parametrized over a disk and scaled down with a dilation to be glued on the branch point of a branched Willmore surface (the limit surface: for instance, an inverted Chen-Gackstatter torus as represented in figure \ref{scpm}, but one could also imagine an inverted L\'opez surface).

We now state our main result.
\begin{theo}
\label{convminibubbleintro}
Let  $\Phi^\varepsilon \,: \, \D \rightarrow \R^3$ a sequence of conformal, weak, Willmore immersions , of Gauss map $\n^\varepsilon$, mean curvature $H^\varepsilon$ and conformal factor $\lambda^\varepsilon$, of parameter $\varepsilon>0$.
We assume 
\begin{enumerate}
\item
\label{hypothesis1}
$\int_\D \left| \nabla \n^\varepsilon \right|^2 dz \le M < \infty$,
\item 
\label{hypothesis2}
$\left\| \nabla \lambda^\varepsilon \right\|_{L^{2, \infty} \left( \D \right) } \le M$,
\item
\label{hypothesis3}
$\displaystyle{\lim_{R \rightarrow \infty}  \left( \lim_{\varepsilon \rightarrow 0} \int_{\D_{\frac{1}{R}} \backslash \D_{\varepsilon R} } \left| \nabla \n^\varepsilon \right|^2 dz \right) = 0}$,
\item
\label{hypothesis4}
$\Phi^\varepsilon \rightarrow \Phi^0$ in $C^\infty_{loc} \left( \D \backslash \{ 0 \} \right)$, with $\Phi^0$ a branched Willmore immersion on $\D$,
\item
\label{hypothesis5}
There exists $C^\varepsilon >0 $ such that
$\frac{\Phi^\varepsilon \left( \varepsilon . \right) - \Phi^\varepsilon (0) }{C^\varepsilon} \rightarrow \Psi $ in $C^\infty_{loc} \left( \C \right)$, with $\Psi$ a minimal immersion (that is of mean curvature $H_\Psi= 0$). 
\end{enumerate}
Then 
$\Phi^\varepsilon \rightarrow \Phi^0$ $C^{2,\alpha} \left( \D \right)$ for all $\alpha <1$.
\end{theo}
Such assumptions are natural if we consider sequences of Willmore immersions of a compact Riemann surface with uniformly bounded total curvature and such that the conformal class of the induced metric is in a compact of the moduli space. Indeed, thanks to theorem \ref{energyquandberriv}, such a sequence $\xi^k$ converges smoothly away from concentration points. 
  Then, in a conformal chart centered on such a point, $\xi^k$ yields a sequence of conformal, weak Willmore immersions $\Phi^k \,: \, \D \rightarrow \R^3$ converging smoothly away from the origin (i.e. hypothesis \ref{hypothesis4}). Hypotheses \ref{hypothesis1} and  \ref{hypothesis2} stand if we choose proper conformal charts (see theorem \ref{theooncontrolelegraddufactconf} below for more details). 
Hypothesis \ref{hypothesis5} then simply specifies that we consider the case where there is only one simple minimal bubble which concentrates at an $\varepsilon_k$ scale. For simplicity's sake we have reparametrized our sequence of immersion by the concentration scale $\varepsilon$. Hypothesis \ref{hypothesis3} is then 
inequality VIII.8 in \cite{bibenergyquant}. An immediate corollary is  the following convergence theorem, which is an improvement over previous convergence results.

\begin{figure}[!h]
\centering
\begin{tikzpicture}[scale=0.75]
\begin{scope}[xshift=-3.4cm, yshift=5cm]
  \begin{axis}[axis lines=none, 
    xmin=-200,xmax=200,
    ymin=-200,ymax=200,
    zmin=-200,zmax=200]
     \addplot3[
         surf,
         colormap/cool,
         domain=0:9,
         y domain=0:2*pi,
         z buffer=sort
       ]
       ( { x*cos(deg(y)) -(x^3 *cos(3*deg(y)))/3}, 
         { -x*sin(deg(y)) -(x^3 *sin(3*deg(y)))/3}, 
         {(x^2)*cos(2*deg(y))}
       );
  	\end{axis}
\end{scope}
\draw [->] (0,6.05) -- (0,3);
\draw (1, 4.5) node {$\varepsilon^3 E\left( \frac{.}{\varepsilon} \right)$};
\draw (0,2.5) circle[radius=2pt];
 \fill (0,2.5) circle[radius=2pt];
\draw (0, 2.25) node[below] {triple branch point};
\begin{scope}[xshift=-3.40cm, yshift=-0.05cm]
  \begin{axis}[axis lines=none, 
    xmin=-20,xmax=20,
    ymin=-20,ymax=20,
    zmin=-20,zmax=20]
     \addplot3[
         surf,
         colormap/cool,
         samples=40,
         domain=0:9,
         y domain=0:2*pi,
         z buffer=sort
       ]
       ( {1/100* x*cos(deg(y)) -1/100*(x^3 *cos(3*deg(y)))/3}, 
         { -1/100*x*sin(deg(y)) -1/100*(x^3 *sin(3*deg(y)))/3}, 
         {1/100*(x^2)*cos(2*deg(y))}
       );
  	\end{axis}
\end{scope}
\draw (-6.5,0) .. controls (-6.5,2) and (-1.5,2.5) .. (0,2.5);
\draw[xscale=-1] (-6.5,0) .. controls (-6.5,2) and (-1.5,2.5) .. (0,2.5);
\draw[rotate=180] (-6.5,0) .. controls (-6.5,2) and (-1.5,2.5) .. (0,2.5);
\draw[yscale=-1] (-6.5,0) .. controls (-6.5,2) and (-1.5,2.5) .. (0,2.5);
\draw (-2,.2) .. controls (-1.5,-0.3) and (-1,-0.5) .. (0,-.5) .. controls (1,-0.5) and (1.5,-0.3) .. (2,0.2);
\draw (-1.75,0) .. controls (-1.5,0.3) and (-1,0.5) .. (0,.5) .. controls (1,0.5) and (1.5,0.3) .. (1.75,0);
\end{tikzpicture}
\caption{Desingularizing the inversion of a Chen-Gackstatter surface with a piece of Enneper.}
\label{scpm}
\end{figure}
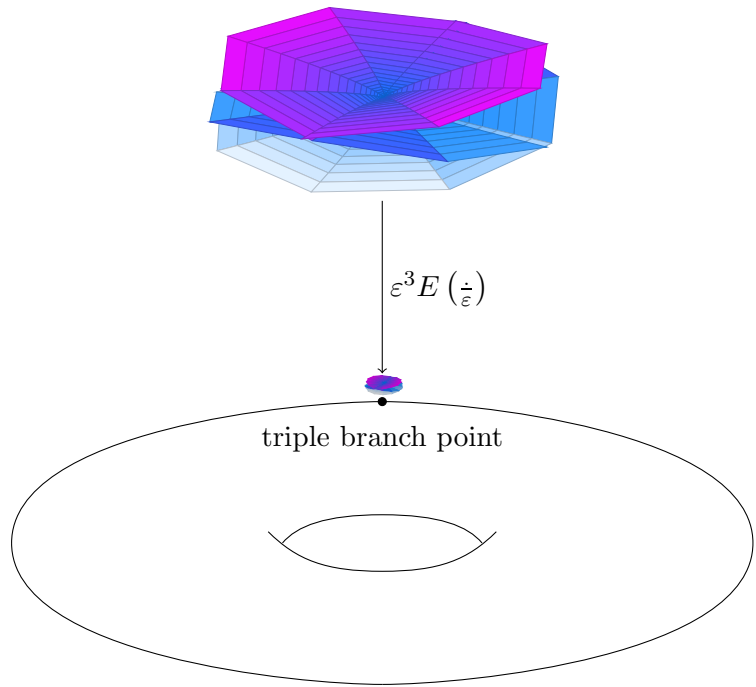

\begin{figure}[!h]
\centering
\begin{tikzpicture}[scale=0.75]
\begin{scope}[xshift=-.25cm,yshift=2.6cm]
\draw (0.5, .5) arc (180:270: 0.5);
\draw (-0.5, 0) arc (270:360: 0.5);
\draw (2,.5) arc (0:90: 0.5);
\draw (1.5,1.5)  arc (270:360: 0.5);
\begin{scope}[xshift=0.5cm]
\draw[rotate=180] (0.5, -2) arc (180:270: 0.5);
\draw[rotate=180] (-.5, -2.5) arc (270:360: 0.5);
\draw[rotate=180] (2,-2) arc (0:90: 0.5);
\draw[rotate=180] (1.5,-1)  arc (270:360: 0.5);
\end{scope}
\draw  (0.5, 2) .. controls (0.5,1.5) and (0.5,1.5) .. (1.5, 1.5);
\draw  (0, 2) .. controls (0,1.5) and (0,1.5) .. (-1, 1.5);
\draw  (0, 0.5) .. controls (0,1) and (0,1) .. (-1, 1);
\draw[dotted] (1,1) arc (-90:90: 0.05 and 0.25);%
\draw (1,1.5) arc (90:270: 0.05 and 0.25);
\draw[dotted] (-.5,1) arc (-90:90: 0.05 and 0.25);
\draw (-.5,1.5) arc (90:270: 0.05 and 0.25);
\draw  (0.5, 0.5) .. controls (0.5,1) and (0.5,1) .. (1.5, 1);
\draw[dotted] (0.5,0.5) arc (0:180: 0.25 and 0.05);
\draw (0,0.5) arc (180:360: 0.25 and 0.05);
\draw[dotted] (0.5,2) arc (0:180: 0.25 and 0.05);
\draw (0,2) arc (180:360: 0.25 and 0.05);
\draw (-1.6,2.15) arc(45:315:1.2);
\draw(-1.6,2.15) arc(90:270:0.05 and 0.85);
\draw[dotted] (-1.6,.45) arc(-90:90:0.05 and 0.85);
\draw[rotate=180] (-2.1,-0.4) arc(45:315:1.2);
\draw[rotate=180] (-2.1,-0.4) arc(90:270:0.05 and 0.85);
\draw[rotate = 180, dotted] (-2.1,-2.13) arc(-90:90:0.05 and 0.85);
\draw[rotate=270] (-2.7,1.1) arc(45:315:1.2);
\draw[rotate=270] (-2.7,1.1) arc(90:270:0.05 and 0.85);
\draw[rotate = 270, dotted] (-2.7,-0.6) arc(-90:90:0.05 and 0.85);
\end{scope}
\draw (-3.5,0) .. controls (-3.5,2) and (-1.5,2.5) .. (0,2.5);
\draw[xscale=-1] (-3.5,0) .. controls (-3.5,2) and (-1.5,2.5) .. (0,2.5);
\draw[rotate=180] (-3.5,0) .. controls (-3.5,2) and (-1.5,2.5) .. (0,2.5);
\draw[yscale=-1] (-3.5,0) .. controls (-3.5,2) and (-1.5,2.5) .. (0,2.5);
\draw (-2,.2) .. controls (-1.5,-0.3) and (-1,-0.5) .. (0,-.5) .. controls (1,-0.5) and (1.5,-0.3) .. (2,0.2);
\draw (-1.75,0) .. controls (-1.5,0.3) and (-1,0.5) .. (0,.5) .. controls (1,0.5) and (1.5,0.3) .. (1.75,0);
\draw (0,2.5) circle[radius=2pt];
 \fill (0,2.5) circle[radius=2pt];
\end{tikzpicture}
\caption{ Using a minimal surface to glue a Clifford torus and $3$ spheres.}
\label{scpmbis}
\end{figure}
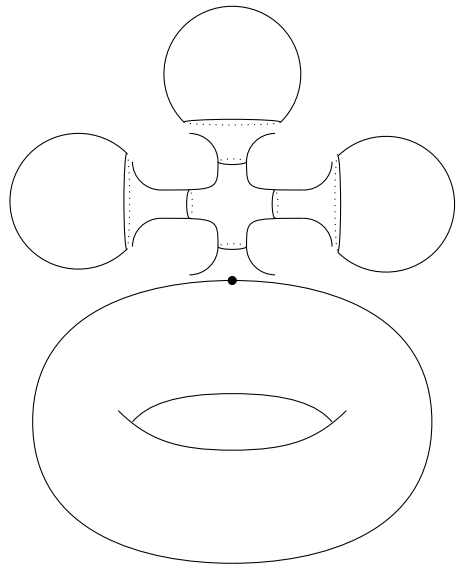

\begin{cor}
\label{lecorintro}
Let $\Phi_k$ be a sequence of Willmore immersions of a closed surface $\Sigma$ satisfying the hypotheses of theorem \ref{energyquandberriv}. We further assume that at each concentration point, a single minimal bubble is blown.
Then 
$\Phi_k \rightarrow \Phi^0$ $C^{2,\alpha} \left( \Sigma \right)$ for all $\alpha<1$.
\end{cor}

In theorem  \ref{convminibubbleintro} and corollary \ref{lecorintro}, the convergence is weaker than in theorem \ref{energyquandberriv}, because they present a convergence result \emph{across the concentration points}.  The convergence is of course still smooth away from the concentration points. However, while at these points the surfaces behave more smoothly than expected, we cannot a priori deduce a $C^\infty$ convergence across them. Indeed, we will show below that the structure of the equations (weak controls on $\Ar$) limits us to a $W^{3, p}$ control on $\Phi$, from which the $C^{2, \alpha}$ result is deduced (see \eqref{lestimeefinaleenPhi} below). While the $W^{3, p}$ control comes naturally in the weak framework,  we chose to present a $C^{2,\alpha}$ convergence result, as it seems more telling in such a geometric situation.

Theorem \ref{convminibubbleintro} and corollary \ref{lecorintro} are to be viewed in the context of other studies of Willmore compactness. 
As has been mentioned, simple bubbling is necessarily non-compact. Further, in \cite{MR3843372}, the authors proved compactness below  for Willmore immersions satisfying $W<12\pi$. The main candidate to realize this threshold would be the previously described  scaled down simple Enneper bubble glued on the branch point of a  Chen-Gackstatter torus (see figure \ref{scpm}). Theorem \ref{convminibubbleintro} represents a first step in understanding such bubbling. If one could prove it cannot occur, the compactness ceiling would then be increased.  
\newline
\textbf{ Added in Proof} During the reviewing process, the author applied and furthered these studies  (in \cite{10.1093/imrn/rnaa079}). He proved that the configuration displayed in figure \ref{scpm} was impossible, and that the regularity found increased the compactness ceiling to sequences of energy below or equal $12\pi$. On the other hand, the  first part of \cite{10.1093/imrn/rnaa079} displays an explicit example of Willmore minimal bubbling (using an Enneper surface to desingularize the branch point of an inverted L\'opez surface) which can illustrate this article's studies.

Theorem \ref{convminibubbleintro} will be proven through a modification of theorem \ref{epsilonregularityRivierecontrolphi}. In the case of minimal bubbling, $\nabla \n$ concentrates, but $H\nabla \Phi$ does not. We will then aim to prove an $\varepsilon$-result  replacing the small total curvature control of theorem \ref{epsilonregularityRivierecontrolphi} by a small Willmore energy control.
Studying the proof of theorem \ref{epsilonregularityRivierecontrolphi} reveals that hypothesis (\ref{petitenenergydanslepsilonregularite}) is used twice. 

 The first time is to show a Harnack inequality on the conformal factor $\lambda$ (following work from F. Hélein, see \cite{bibharmmaps}, or \cite{bibmulsve} for a different treatment by S. M{\"u}ller and V. {\v S}ver{\'a}k), and deduce a $L^{2,\infty}$ control on the first Willmore quantity $\vec{L}$. This Harnack inequality stems from putting the classical Liouville equation in divergence form with a local Coulomb frame, and applying Wente's lemmas. Controlling this frame requires a small estimate on $\nabla \n$ that cannot be avoided. However, this can be done with some flexibility. For instance, on disks of bounded (not necessarily small) $\nabla \n$ energy, one can extend these results up to counting the number of small energy disks needed to cover the domain. To this end, we introduce 
\begin{equation}\label{ler0}r_0 =\frac{1}{\rho} \inf \left\{ s \left| \int_{B_s(p)} \left| \, \nabla \n \right|^2 = \frac{4 \pi}{3}, \: \forall p \in \D_\rho \text{ s.t. } B_s(p) \subset \D_\rho \right. \right\}.\end{equation}
This parameter marks how relatively small a ball has to be to ensure that it does not contain too much energy, and its inverse will bound the number of balls with small energy covering the disk. Alternatively, in the framework of theorem \ref{energyquandberriv}, it measures how concentrated $\nabla \n$ is on a disk.

The second use of hypothesis (\ref{petitenenergydanslepsilonregularite}) lies in the exploitation of the peculiar Jacobian form of system (\ref{systemenRSPhiannexi}) to break its criticality. 
We will show in this article  that it can be rewritten into 
\begin{equation}
\label{systemenRSPhiannexmieux}
\left\{
\begin{aligned}
\Delta S &=  \left\langle H \nabla \Phi, \nabla^\perp \vec{R} \right\rangle \\
\Delta \vec{R} &= - H \nabla \Phi \times \nabla^\perp \vec{R} - \nabla^\perp S H \nabla \Phi\\
\Delta \Phi &= \frac{1}{2} \left( \nabla^\perp S. \nabla \Phi + \nabla^\perp \vec{R} \times \nabla \Phi \right).
\end{aligned}
\right.
\end{equation}
This new equivalent system, along with some tight estimates in Lorentz spaces will yield an $\varepsilon$-regularity result with a small Willmore energy hypothesis.
\begin{theo}
\label{epsilonregularitehnablaphifaibleintro}
Let $\Phi \in \mathcal{E}\left(\D \right) $ satisfy the hypotheses of theorem \ref{epsilonreularitylintro}.
Then there exists $\varepsilon_0'$  depending only on $C_0$ such that if $$  \left\| H \nabla \Phi \right\|_{L^2 \left( \D \right)}  \le \varepsilon_0',$$ then for any $r< 1$ there exists a constant  $ C \in \R$  depending on  $r$, $C_0$, $p$ and  $r_0$ (defined in (\ref{ler0})) such that 
$$\left\| H \nabla \Phi \right\|_{L^\infty \left( \D_{ r  } \right)} \le  C { \left\| H \nabla \Phi \right\|_{L^2 \left( \D \right)}},$$ 
and 
$$\left\| \nabla \Phi \right\|_{W^{2,p} \left( \D_{ r  } \right)} \le     C \| \nabla \Phi \|_{L^2 \left( \D \right) } $$
for all $p<\infty$.
\end{theo}
Theorem \ref{epsilonregularitehnablaphifaibleintro} as stated makes use of the parameter $r_0$, and can be applied when $r_0$ is bounded from below, but degenerates as soon as this is no longer the case. However this dependance on $r_0$ only appears as an artefact of an estimate on $\vec{L}$ (see theorem \ref{n21controlintro} below). In fact, we will prove a less immediately eloquent but more adaptable result (see theorem \ref{epsilonregularitehnablaphifaibleLL2inftyi} below).

One has to be aware that estimates in $r_0$ will not enable us to prove theorem \ref{convminibubbleintro}. Indeed, as the energy concentrates, $r_0 \simeq \varepsilon$ goes to $0$.  However, applied to a ball of radius $\varepsilon$, theorem \ref{epsilonreularitylintro} will yield uniform estimates for $\Lr$ on the bubble $\D_\varepsilon$. One then only has to control $\Lr$ on the so-called "neck area": $\D \backslash \D_{\varepsilon}$, and combine these two to control $\Lr$ across the concentration point. This estimate will then be exploited through the equations to obtain the result.

In section \ref{lasection1} we will recall the notion of weak Willmore immersions and prove generic controls on $\Lr$, $H \nabla \Phi$ and $\nabla \n$ in Lorentz spaces. Section \ref{lasection2} will be devoted to the proof of theorem  \ref{epsilonregularitehnablaphifaibleintro} (and its more general version  \ref{epsilonregularitehnablaphifaibleLL2inftyi}) while  section \ref{lasection3} will focus on controlling $\Lr$ on annuli of degenerating conformal classes. We will conclude in section \ref{lasection4} with the proof of theorem \ref{convminibubbleintro}.

\section{First regularity results for weak Willmore immersions }
\label{lasection1}
\subsection{Weak Willmore immersions of a surface}
Let $\Sigma$ be an arbitrary closed compact two-dimensional manifold.
Let $g_0$ be a smooth "reference" metric on $\Sigma$. The Sobolev spaces $W^{k,p} \left( \Sigma, \R^3 \right)$ of measurable maps from $\Sigma$ into $\R^3$ is defined as 
$$W^{k,p} \left( \Sigma , \R^3 \right):= \left\{ f \text{ measurable: } \Sigma \rightarrow \R^3 \text{ s.t } \sum_{l= 0}^{k} \int_\Sigma \left| \nabla_{g_0}^l f \right|^p_{g_0} d\mathrm{vol}_{g_0} < \infty \right\}.$$
Since $\Sigma$ is assumed to be compact, this definition does not depend on $g_0$.

We will work with the concept of weak immersions introduced by T. Rivière, which represent the correct starting framework for studying Willmore immersions. One might notice the presentation of this notion has changed through the years (compare definition I.1 in \cite{bibanalysisaspects} with its equivalent in subsection 1.2 in \cite{MR3843372}). While we use the latter, which is sufficient for our needs, one could take slightly less demanding (albeit more complex) starting hypotheses.
\begin{de}
\label{weakimmersions}
Let $\Phi \,: \, \Sigma \rightarrow \R^3$. Let $g_\Phi = \Phi^* \xi$ be the first fundamental form of $\Phi$ and $\n$ its Gauss map.  Then $\Phi$ is called a weak immersion with locally $L^2$-bounded second fundamental form if $\Phi \in W^{1,\infty} \left( \Sigma \right)$, if there exists a constant $C_\Phi$ such that 
$$ \frac{1}{C_\Phi} g_0  \le g_\Phi  \le C_\Phi g_0,$$
and if $$\int_\Sigma \left| d \n \right|^2_{g_\Phi} d \mathrm{vol}_\Phi < \infty.$$ The set of weak immersions with $L^2$-bounded second fundamental form on $\Sigma$ will be denoted $\mathcal{E}(\Sigma)$.
\end{de}
One of the advantages of such weak immersions is that they allow us to work with conformal maps as shown by theorem 5.1.1 of \cite{bibharmmaps}.
\begin{theo}
\label{localconformalcoordinates}
Let $\Phi$ be a weak immersion from $\Sigma$ into $\R^3$ with $L^2$-bounded second fundamental form. Then for every $x \in \Sigma$, there exists an open disk $ D$ in $\Sigma$ containing $x$ and a homeomorphism $\Psi \,: \, \D \rightarrow D$ such that  $\Phi \circ \Psi$ is a conformal bilipschitz immersion. The induced metric $g = \left(\Phi \circ \Psi \right)^* \xi$
 is continuous. 
\end{theo}
Further, proving estimates on the Greeen function of $\Sigma$, P. Laurain and T. Rivière have shown in theorem 3.1 of  \cite{biboptimalestimates} that up to chosing a specific atlas, one could have further control on the conformal factor.
\begin{theo}
\label{theooncontrolelegraddufactconf}
Let $(\Sigma, g)$ be a closed Riemann surface of fixed genus greater than one. Let $h$ denote the metric with constant curvature (and volume equal to one in the torus case) in the conformal class of $g$ and $\Phi \in \mathcal{E}(\Sigma)$ conformal, that is: 
$$\Phi^* \xi = e^{2u} h.$$ 
Then, there exists a finite conformal atlas  $(U_i, \Psi_i)$ and a positive constant $C$ depending only on the genus of $\Sigma$, such that 
$$ \left\| \nabla \lambda_i \right\|_{L^{2, \infty} \left( V_i \right)} \le  C \left\| \nabla_{\Phi^* \xi} \n \right\|^2_{L^2 \left( \Sigma \right)},$$
with $\lambda_i = \frac{1}{2} \log \frac{ \left| \nabla \Phi \right|^2}{2}$ the conformal factor of $\Phi \circ \Psi_i^{-1}$ in $V_i = \Psi_i (U_i)$.
\end{theo}
Thus, given $\tilde \Phi \in \mathcal{E}\left( \Sigma \right)$  we can choose a conformal atlas such that, in a local chart on $\D$ of this atlas, $\tilde \Phi$ yields  $\Phi \in \mathcal{E} \left( \D \right)$ satisfying 
\begin{equation}
\label{hypothesecontroleconformalfacteur2infinietphiw1infini}
\begin{aligned}
&\left\| \nabla \lambda \right\|_{L^{2, \infty} \left( \D \right) } \le C_0.
\end{aligned}
\end{equation}
One can then systematically study any $\tilde \Phi \in \mathcal{E} \left(\Sigma \right)$ in such  local conformal charts, as a conformal bilipschitz map $\Phi \in \mathcal{E}\left(\D\right)$ satisfying (\ref{hypothesecontroleconformalfacteur2infinietphiw1infini}).

We can now introduce the notion of weak Willmore immersions (definition I.2 in \cite{bibanalysisaspects}).
\begin{de}
 Let $\Phi \in \mathcal{E} \left( \Sigma \right)$. $\Phi$  is a weak Willmore immersion if 
\begin{equation} \label{equationwillmorefaible} \mathrm{div}\left(  \nabla \vec{H} -3 \pi_{\n} \left( \nabla \vec{H} \right) + \nabla^\perp \n \times \vec{H} \right) = 0 \end{equation}
holds in a distributional sense in every conformal parametrization $\Psi  \,: \, \D \rightarrow D$ on every neighborhood $D$ of $x$ , for all $x \in \Sigma$. Here the operators $\mathrm{div}$, $\nabla$ and $\nabla^\perp = \begin{pmatrix} - \partial_y \\ \partial_x \end{pmatrix} $ are to be understood with respect to the flat metric on $\D$.
\end{de}

Equation (\ref{equationwillmorefaible}) is in fact the classical Willmore equation (\ref{lequationdewillmoreclassique}) in divergence form.
\begin{equation}
\label{lequationdewillmoreclassique}
\Delta H + \big| \Ar \big|^2 H = 0.
\end{equation}
Immersions satisfying (\ref{lequationdewillmoreclassique}) are called Willmore immersions.
The weak Willmore equation was introduced to work with weak immersions since (\ref{equationwillmorefaible}) requires less regularity than (\ref{lequationdewillmoreclassique}). However, a consequence of  theorem I.5 in \cite{bibanalysisaspects} is that weak Willmore immersions are smooth, and necessarily Willmore immersions.

\subsection{Harnack inequalities on the conformal factor}
Works by F. Hélein ensured that in disks of small energy, and that, up to a reasonable  (see (\ref{hypothesecontroleconformalfacteur2infinietphiw1infini})) assumption on $\left\| \nabla \lambda \right\|_{L^{2, \infty} \left( \D \right) }$, the conformal factor could be controlled pointwise. We here give a version from theorem 5.5 of \cite{bibpcmi}.

\begin{theo}
\label{controlconformalfactor}
Let $\Phi \in \mathcal{E}\left(\D \right)$, conformal. Let  $\n$ be its Gauss map and $\lambda$ its conformal factor.
We assume $$ \int_{\D}  \left| \nabla \n \right|^2 < \frac{8 \pi}{3},$$
and 
\begin{equation}
\label{hypothesecontroleconformalfacteur2infini}
\left\| \nabla \lambda \right\|_{L^{2,\infty} \left( \D \right)} \le C_0.
\end{equation}
Then, for any $r<1$ there exists $c \in \R$ and $C \in \R$ depending on $r$ and $C_0$ such that 
$$ \left\| \lambda - c \right\|_{L^\infty \left( \D_r \right) } \le C .$$
\end{theo}
This theorem can be adapted to disks of arbitrary radii  without losing control on the constant.
\begin{cor}
\label{controlconformalfactorcor}
Let $\Phi \in \mathcal{E}\left(\D_\rho \right)$, conformal. Let $\n$ be its Gauss map and $\lambda$ its conformal factor.
We assume  $$ \int_{\D_\rho}  \left| \nabla \n \right|^2 < \frac{8 \pi}{3}$$
and 
$$\left\| \nabla \lambda \right\|_{L^{2,\infty} \left( \D_\rho \right)} \le C_0 .$$
Then, for any $r<1$ there exists $c_\rho \in \R$ and $C \in \R$ depending on $r$ and $C_0$ such that  
$$ \left\| \lambda - c_\rho \right\|_{L^\infty \left( \D_{r \rho} \right) } \le C.$$
\end{cor}
\begin{proof}
Let $\Phi_\rho = \Phi \left( \rho . \right)$,
 $\n_\rho$ be its Gauss map and  $\lambda_\rho$ its conformal factor. Straightforward computations yield 
\begin{equation}
\label{equationauxiliairerescalinglamndacor}
 \begin{aligned}
e^{\lambda_\rho} &= \rho e^{\lambda} \left( \rho. \right) 
\end{aligned}
\end{equation}
and
\begin{equation}
\label{equationauxiliairerescalinglamndacorbis}
\begin{aligned}
\n_\rho &= \n \left( \rho. \right).
\end{aligned}
\end{equation}
Then $$\int_{\D} \left| \nabla \n_\rho \right|^2dz = \int_{\D_\rho}  \left| \nabla \n \right|^2 dz < \frac{8 \pi}{3}$$ 
and, thanks to (\ref{equationauxiliairerescalinglamndacor}),
$$ \| \nabla \lambda_\rho \|_{L^{2, \infty} \left( \D \right) }  = \| \nabla \left(  \lambda ( \rho . ) + \ln \rho \right) \|_{L^{2, \infty} \left( \D \right) } = \| \nabla \lambda \|_{L^{2, \infty} \left( \D_\rho \right) } \le C_0 $$ owing to the scaling-invariance properties of the $L^2$ and $L^{2, \infty}$ norms.
Applying theorem \ref{controlconformalfactor} one finds there exists $c \in \R$ and $C \in \R$ depending on $r$ and $C_0$ such that 
$$ \left\| \lambda_\rho - c \right\|_{L^\infty \left( \D_{r} \right) } \le C .$$
However, using (\ref{equationauxiliairerescalinglamndacor}),
$$ \left\| \lambda - c_\rho \right\|_{L^\infty \left( \D_{r \rho} \right) } \le C$$
with $c_\rho = c - \ln \rho$ and the same $C$.
\end{proof}

We can extend the control to domains with merely $\int_\D \left| \nabla \n \right|^2 < \infty$ up to adding an additionnal parameter $r_0$ to the constant. As explained in the introduction, $r_0$ measures how uniformly small a ball in the disk has to be to have sufficiently small $\nabla \n$ energy and thus in turn how many of these small balls are needed to cover the domain of study. We recall the definition of $r_0$ before proceeding:
$$r_0 = \frac{1}{\rho} \inf \left\{ s \left| \int_{B_s(p)} \left| \, \nabla \n \right|^2 = \frac{4 \pi}{3}, \: \forall p \in \D_\rho \text{ s.t. } B_s(p) \subset \D_\rho \right. \right\}.$$
\begin{cor}
\label{controlconformalfactorsanspetit}
Let $\Phi \in \mathcal{E}\left(\D_\rho \right) $ conformal,  let $\n$ be its Gauss map and $\lambda$ its conformal factor.
We assume that $$ \left\| \nabla \lambda \right\|_{L^{2,\infty} \left( \D_\rho \right)} + \left\| \nabla \n \right\|_{L^{2} \left( \D_\rho \right)} \le C_0 .$$
Then, for any $r<1$ there exists $c_\rho \in \R$ and $C \in \R$ depending on  $r$, $C_0$ and $r_0$  such that 
$$ \left\| \lambda - c_\rho \right\|_{L^\infty \left( \D_{ r \rho} \right) } \le C. $$
\end{cor}
\begin{proof}
We prove the result on $\D$, then working as in the proof of corollary \ref{controlconformalfactorcor} we can extend the result to $\D_\rho$.

If  $ \int_\D \left| \nabla \n \right|^2  < \frac{8 \pi}{6}$ then one can simply apply theorem \ref{controlconformalfactor}. 
Else let $r<1$, and $r_1 = \min \left( \frac{1-r}{2}, r_0 \right)$. We cover $\D_r$ with a finite number $N(r_1)$ of open disks $ \left(B_{\frac{r_1}{2}} (p_i) \right)_{i=1 \dots n } $.

One can then apply 
corollary \ref{controlconformalfactorcor} on each $B_{r_1}\left( p_i \right)$ and find $c_i \in \R$ such that 
\begin{equation} \label{controlelambdaci} \left\| \lambda - c_i  \right\|_{L^\infty \left( B_{\frac{r_1}{2}} \left( p_i\right) \right)} \le C . \end{equation}
Here $C$ is a constant depending only on $C_0$. 
 Let $i,j \in I$ such that $ B_{\frac{r_1}{2}} \left( p_i\right) \cap  B_{\frac{r_1}{2}} \left( p_j\right) \neq \emptyset$. Then 
\begin{equation}
\label{deuxccontigus}
\begin{aligned}
\left| c_i - c_j  \right| &\le \left| c_i - \lambda (x) \right| + \left| c_j - \lambda (x) \right| \\
&\le  \left\| \lambda - c_i  \right\|_{L^\infty \left( B_{\frac{r_1}{2}} \left( p_i\right) \right)} +  \left\| \lambda - c_j \right\|_{L^\infty \left( B_{\frac{r_1}{2}} \left( p_j\right) \right)} \\
&\le 2   C.
\end{aligned}
\end{equation}
Taking any $i,j \in I$, let $\gamma_{ij}$ be a straight line linking any fixed $x_i \in  B_{\frac{r_1}{2}} \left( p_i\right)$ to any fixed $x_j \in  B_{\frac{r_1}{2}} \left( p_j\right)$. $\gamma_{ij}$ goes through the disks $ \left( B_{\frac{r_1}{2}} \left( p_{q_l}\right) \right)_{q_l \in J \subset I}$, ordered such that  $$ B_{\frac{r_1}{2}} \left( p_{q_l}\right) \cap  B_{\frac{r_1}{2}} \left( p_{q_{l+1}}\right) \neq \emptyset.$$
Then, thanks to  (\ref{deuxccontigus}), $$ \begin{aligned}
\left| c_i - c_j  \right| &\le \sum_{l} \left| c_{q_l} - c_{q_{l+1}} \right| \\
&\le  \sum_{l}  2   C \\
&\le 2N  C ,
\end{aligned}
$$
since $ \gamma_{ij}$ goes through at most $N $ disks.

Setting $c = c_1$, one deduces 
\begin{equation}
\label{controledesc}
\left| c - c_i \right| \le  2N  C  \quad \forall i \in I.
\end{equation}
Then given any $x \in \D_r$ we find a $i \in I$ such that $x \in B_{\frac{r_1}{2}} ( p_i)$ and have, using (\ref{controlelambdaci}) and (\ref{controledesc}),
$$\left| \lambda(x) - c \right| \le \left| \lambda(x) - c_i\right| + \left| c - c_i \right| \le \left( 2N+1 \right)  C.  $$
Taking the supremum over $x$ we conclude with 
$$ \left\| \lambda - c \right\|_{L^\infty \left( \D_r \right) } \le  \left( 2N+1 \right) C$$ which is as announced given that  $N$ depends only on $r$ and $r_0$.
\end{proof}

This Harnack inequality ensures that (\ref{equationwillmorefaible}) has a distributional meaning in conformal maps. Indeed, if we consider  $\Phi \in \mathcal{E}\left( \D \right) $ satisfying hypothesis (\ref{hypothesecontroleconformalfacteur2infini}), $\nabla \n \in L^2(\D)$ and its respective tracefull and tracefree part $H\nabla \Phi$ and $\Ar \nabla \Phi$ are properly defined as $L^2(\D)$ functions (see  (\ref{HnablaphietArnablaphienfonctionden}) for details). Thus, corollary \ref{controlconformalfactorsanspetit}  ensures that,  for any $r<1$,  there exists $\Lambda \in \R$ such that on $\D_r$ we have 
\begin{equation}
\label{inequalityharnacl}
\frac{e^\Lambda}{C} \le e^\lambda \le C e^\Lambda.
\end{equation}
Hence, since $\left| H \right| = e^{-\lambda} \left|H \nabla \Phi \right|$ we have on $\D_r$ 
\begin{equation} \label{hhasameaninginaconformaldisk} \begin{aligned} \| H \|_{L^2 \left( \D_r \right) } &\le e^{-\Lambda} C  \| H \nabla \Phi \|_{L^2 \left( \D \right) } \\ &\le e^{-\Lambda}  C  \| \nabla \n \|_{L^2 \left( \D \right)}< + \infty. \end{aligned}\end{equation}
As a result, (\ref{equationwillmorefaible}) is well-defined in the distributional sense, which will allow us to introduce  divergence free quantities for the Willmore equations.

\subsection{Divergence free vector fields for the Willmore immersions}
\label{subsectionWillmorequantities}
As said in the introduction, T. Rivière has defined auxiliary quantities (theorem I.4 in \cite{bibanalysisaspects}) playing a crucial part in the regularity of Willmore surfaces. We recall their definition before any further exploitation.

\begin{de}
Let $\Phi \in \mathcal{E} \left( \D \right)$ be a weak  Willmore immersion.
Then, there exists $\Lr \in \mathcal{D}' \left( \D \right)$ 
such that 
 \begin{equation} \label{definitionl} \nabla^\perp \Lr = \nabla \vec{H} -3 \pi_{\n} \left( \nabla \vec{H} \right) + \nabla^\perp \n \times \vec{H}.\end{equation}
In the following, we will call $\Lr$ the first Willmore quantity.
\end{de}
\begin{prop} 
Let $\Phi \in \mathcal{E} \left( \D \right) $ be a weak  Willmore immersion.
Then, for any $\Lr \in \mathcal{D}' \left( \D \right)$ satisfying  (\ref{definitionl}) we have
$$\begin{aligned}
&\mathrm{div} \left( \langle \Lr, \nabla^\perp \Phi \rangle \right) = 0 \\
& \mathrm{div} \left( \Lr \times \nabla^\perp \Phi + 2H \nabla^\perp \Phi \right) = 0.
\end{aligned}$$
Thus, there exists $S$ and $\vec{R}  \in \mathcal{D}' \left( \D \right)$ such that
 \begin{equation} \label{definitionSR} \begin{aligned}  &\nabla^\perp S = \langle \Lr, \nabla^\perp \Phi \rangle \\
& \nabla^\perp \vec{R} = \Lr \times \nabla^\perp \Phi + 2H \nabla^\perp \Phi . \end{aligned}\end{equation}
In the following, we will call $S$ and $\vec{R}$ the second and third Willmore quantity.
\end{prop}
We remark that $\Lr$, $\vec{R}$ and $S$ are defined up to a constant that we can (and will) adjust. 

The key role played by $S$ and $\vec{R}$ revolves around the system of equations they satisfy (as stated by theorem 7.5 and corollary 7.6 of \cite{bibpcmi}).
\begin{theo}
\label{theoremequationRS}
Let $\Phi \in \mathcal{E} \left( \D \right) $ be a weak  Willmore immersion. Then $S$ and $\vec{R}$ satisfy
\label{lesystemeenRS}
\begin{equation}
\label{RetSennvectoriel}
\begin{aligned}
\nabla S &= - \left\langle \n, \nabla^\perp \vec{R} \right\rangle \\
\nabla \vec{R} &= \n \times \nabla^\perp \vec{R} + \nabla^\perp S \n,
\end{aligned}
\end{equation}
and hence
\begin{equation}
\label{systemenRSPhiannex}
\left\{
\begin{aligned}
\Delta S &= - \left\langle \nabla \n, \nabla^\perp \vec{R} \right\rangle \\
\Delta \vec{R} &= \nabla \n \times \nabla^\perp \vec{R} + \nabla^\perp S \nabla \n\\
\Delta \Phi &= \frac{1}{2} \left( \nabla^\perp S. \nabla \Phi + \nabla^\perp \vec{R} \times \nabla \Phi \right).
\end{aligned}
\right.
\end{equation}
\end{theo}
This system can be slightly changed to better suit our needs.
\begin{theo}
\label{theoRSPhi}
Let $\Phi \in \mathcal{E} \left( \D \right) $ be a weak  Willmore immersion.  Then, $S$ and $\vec{R}$ satisfy
\begin{equation}
\label{systemenRSPhiannexmieux}
\left\{
\begin{aligned}
\Delta S &=  \left\langle H \nabla \Phi, \nabla^\perp \vec{R} \right\rangle \\
\Delta \vec{R} &= - H \nabla \Phi \times \nabla^\perp \vec{R} - \nabla^\perp S H \nabla \Phi\\
\Delta \Phi &= \frac{1}{2} \left( \nabla^\perp S. \nabla \Phi + \nabla^\perp \vec{R} \times \nabla \Phi \right).
\end{aligned}
\right.
\end{equation}
\end{theo}
\begin{proof}
Computations are done in the appendix (see section A.2).
\end{proof}

\subsection{Control of $\Lr e^\lambda$ on a disk}
This section is devoted to the following result, which is only a slight improvement over theorem 7.4 of \cite{bibpcmi}, with a control by $H \nabla \Phi$ replacing one by $\nabla \n$. However, we will  follow \emph{mutatis mutandis} the previous proof.
\begin{theo}
\label{epsilonreularitylintro}
Let $\Phi \in \mathcal{E}\left(\D_\rho \right)$ be a conformal weak Willmore immersion.  Let  $\n$ denote its Gauss map, $H$ its mean curvature and $\lambda$ its conformal factor.
We assume $$ \left\| \nabla \lambda \right\|_{L^{2,\infty} \left( \D_\rho \right)}+\| \nabla \n \|_{L^2 \left( \D_\rho \right)} \le C_0.$$
Then for any=
$r< 1$ there exists a constant $\vec{\mathcal{L}} \in \R^3$ and a constant $ C \in \R$ depending on $ r$, $C_0$  and $r_0$ (defined in (\ref{ler0}))=
such that 
$$ \left\| e^{\lambda} \left( \Lr - \vec{\mathcal{L}}  \right) \right\|_{L^{2, \infty} \left( \D_{r\rho} \right) } \le C \left\| H \nabla \Phi \right\|_{L^2 \left( \D_\rho \right) },$$
where $\Lr$ is given by (\ref{LRSintro}).
\end{theo}
\begin{proof}
As before we will  prove the theorem on $\D$. The proof on $\D_\rho$ follows as in corollary \ref{controlconformalfactorcor}.
Let $\Phi \in \mathcal{E}\left(\D \right)$ be a conformal weak Willmore immersion, $\n$ its Gauss map, $H$ its mean curvature and $\lambda$ its conformal factor.
We assume that $$ \left\| \nabla \lambda \right\|_{L^{2,\infty} \left( \D \right)} + \left\| \nabla \n \right\|_{L^{2} \left( \D \right)} \le C_0.$$ Let $r<1$ and $\Lr \in \mathcal{D}' \left( \D \right)$ satisfying (\ref{definitionl}).=
\newline
\underline{ \bf{ Step 1: }} {\bf Control of the conformal factor  } \newline
Applying corollary \ref{controlconformalfactorsanspetit}
 we find $\Lambda \in \R$ and $C$ depending on  $r$, $C_0$ and $r_0$ such that 
$$ \left\| \lambda - \Lambda \right\|_{L^\infty \left( \D_{\frac{r+1}{2}} \right) } \le C.$$
Consequently, $\lambda$ satisfies (\ref{inequalityharnacl}),
$$
\forall x \in \D_{\frac{r+ 1}{2}} \quad \frac{e^\Lambda}{C}  \le e^{\lambda(x)} \le C e^\Lambda.
$$
\newline
\underline{ \bf{ Step 2: }} {\bf Control on $\nabla \Lr$  } \newline
Estimate (\ref{hhasameaninginaconformaldisk}) then stands:   $$\begin{aligned}
 \| H \|_{L^2 \left( \D_{\frac{r+1}{2}} \right)} &\le C e^{-\Lambda}  \| H \nabla \Phi \|_{L^2 \left( \D_{\frac{r+1}{2}} \right)}. \end{aligned}$$
 
We can exploit it to control the right-hand side of  (\ref{definitionl}). First, using the fact that the tangent part of $\nabla \vec{H}$, $ \pi_{T} \left( \nabla \vec{H} \right)$, satisfies $ \pi_{T} \left( \nabla \vec{H} \right) = H \nabla \n$,  we recast  (\ref{definitionl}) as 
\begin{equation}\label{eqauxsurleL}\begin{aligned} \nabla^\perp \Lr &= \nabla \vec{H}-3 \pi_{\n} \left( \nabla \vec{H} \right) + \nabla^\perp \n \times \vec{H} \\ &= \nabla \vec{H}-3\nabla \vec{H}  + 3 \pi_{T} \left( \nabla \vec{H} \right)+ \nabla^\perp \n \times \vec{H} \\
&= - 2 \nabla \vec{H} +3 H \nabla \n  +  \nabla^\perp \n \times \vec{H}. \end{aligned}\end{equation}
Then we control each term of the right-hand side as follows.
Moreover
$$\begin{aligned}
 \| \nabla^\perp \n \times \vec{H} \|_{L^1\left(\D_{\frac{r+1}{2}} \right)} &\le \left\| \nabla \n \right\|_{L^2 \left( \D_{\frac{r+1}{2}} \right)}\|  \vec{H} \|_{L^{2} \left( \D_{\frac{r+1}{2}} \right)} \\&\le    C  e^{-\Lambda} \left\| \nabla \n \right\|_{L^2 \left( \D \right)}  \| H \nabla \Phi  \|_{L^2 \left( \D \right)},
\end{aligned} $$
while
$$ \begin{aligned}
 \| H  \nabla \n  \|_{L^1\left(\D_{\frac{r+1}{2}} \right)} &\le  \left\| \nabla \n \right\|_{L^2 \left( \D_{\frac{r+1}{2}} \right)}\|  \vec{H} \|_{L^{2} \left( \D_{\frac{r+1}{2}} \right)}  \\ &\le  C e^{-\Lambda}  \left\| \nabla \n \right\|_{L^2 \left( \D \right)}   \| H \nabla \Phi  \|_{L^2 \left( \D \right)}.
\end{aligned}$$

  The last three estimates combined give
$$\nabla \Lr \in \nabla^\perp L^2 \left(\D_{\frac{r+1}{2}} \right) \bigoplus L^1 \left( \D_{\frac{r+1}{2}} \right).$$ 

\vspace{5mm}

\underline{ \bf{ Step 3: }} {\bf Conclusion  }

Thanks to Step 2 and  theorem \ref{theoremepourbornerL} (see appendix)

$$ \exists \vec{\mathcal{L}} \in \R^3 \quad  \left\| \Lr - \vec{\mathcal{L}}  \right\|_{L^{2,\infty} \left( \D_r \right)} \le C  e^{-\Lambda}  \| H \nabla \Phi  \|_{L^2 \left( \D \right)}$$
with $C$ a real constant that depends on $r$, $C_0$ and $r_0$. Hence 

$$\begin{aligned} \left\| \left( \Lr -\vec{\mathcal{L}}  \right)e^\lambda \right\|_{L^{2,\infty} \left( \D_r \right) } &\le  e^{\Lambda}  \left\|  \Lr - \vec{\mathcal{L}}   \right\|_{L^{2,\infty}\left( \D_r\right)} \\
&\le C \| H \nabla \Phi  \|_{L^2 \left( \D \right)}, \end{aligned}$$
with $C$ as desired. This concludes the proof on $\D$.
\end{proof}

\begin{remark}
Theorem \ref{theoremepourbornerL} yields, in fact, a $L^2$ inequality. Since $L^2 \subset L^{(2, \infty)}$ the desired (and weaker) control follows. We chose to present a $L^{2, \infty}$ inequality because it is sufficient to recover the regularity (see theorem \ref{epsilonregularitehnablaphifaibleLL2inftyi}), and because  in the bubbling case, which is the objective of this paper, we will start with $L^{2,\infty}$ controls (due to the nature of the neck estimates, see section \ref{lasection3} below). All results will then be presented with this generic  starting estimate on $ \Lr$.
\end{remark}

\subsection{$L^{2,1}$ controls in the generic case}
Without small controls on $H$ or $\n$,  some results can be achieved in term of Lorentz spaces estimates as shown by the following.
\begin{theo}
\label{n21controlintro}
Let $\Phi \in \mathcal{E}\left(\D_\rho \right) $ satisfy the hypotheses of theorem \ref{epsilonreularitylintro}.
Then for any $r< 1$ there exists a constant $ C \in \R$  depending on  $r$, $C_0$ and  $r_0$ (defined in (\ref{ler0})) such that 
$$\|H \nabla \Phi \|_{L^{2,1} \left( \D_{r\rho} \right) } \le C   \| H \nabla \Phi  \|_{L^2 \left( \D_\rho \right)},$$ and
$$\left\| \nabla \n \right\|_{L ^{2,1} \left( \D_{r \rho} \right) } \le C \left\| \nabla \n \right\|_{L ^{2} \left( \D_\rho \right) }.$$
\end{theo}
We first prove a more flexible result than theorem \ref{n21controlintro} (in that it does not reference $r_0$) controlling the $L^{2,1}$ norm of $\nabla \n$ under $L^{2,\infty}$ assumptions on $\Lr$.
\begin{theo}
\label{n21controlsousLL2inftyintro}
Let $\Phi \in \mathcal{E}\left(\D_\rho \right)$ be a conformal weak Willmore immersion, $\n$ its Gauss map, $H$ its mean curvature, $\lambda$ its conformal factor and $\Lr$ its first Willmore quantity.
We assume  $$ \left\| \nabla \lambda \right\|_{L^{2,\infty} \left( \D_\rho \right)}+\| \nabla \n \|_{L^2 \left( \D_\rho \right)} \le C_0 ,$$
 and that there exists $r'<1$ and  and $C_1>0$ such that $$ \left\| \Lr e^{\lambda} \right\|_{L^{2,\infty} \left( \D_{ r' \rho} \right)} \le C_1 \left\| H \nabla \Phi \right\|_{L^2 \left( \D_\rho \right) }.$$
Then, for any $r< r'$ there exists a constant $ C$ depending on $ r$, $r'$, $C_0$ and $C_1$  such that 
$$\|H \nabla \Phi \|_{L^{2,1} \left( \D_{r \rho} \right) } \le C   \| H \nabla \Phi  \|_{L^2 \left( \D_\rho \right)},$$ and
$$\left\| \nabla \n \right\|_{L ^{2,1} \left( \D_{r\rho} \right) } \le C \left\| \nabla \n \right\|_{L ^{2} \left( \D_\rho \right) }.$$
Furthermore, the associated second and third Willmore quantities satisfy 
$$\| \nabla S \|_{L^{2,1} \left( \D_{r \rho} \right) } + \| \nabla \vec{R} \|_{L^{2,1} \left( \D_{r\rho} \right) }\le C   \| H \nabla \Phi  \|_{L^2 \left( \D_\rho \right)}.$$
\end{theo}
\begin{proof}
As before, it is enough to work on the unit disk and conclude with a dilation to obtain the result on disks of arbitrari radii.\newline
\noindent \underline{ \bf{ Step 1: }} {\bf $L^{2,1}$ control of $\nabla S$ and $\nabla \vec{R}$  }\newline
Let $r'<1$ and $\Lr$  such that $$ \left\| \Lr e^{\lambda} \right\|_{L^{2,\infty} \left( \D_{r'} \right)} \le C_1 \left\| H \nabla \Phi \right\|_{L^2 \left( \D \right) }.$$
Then  $S$ and $\vec{R}$ defined as $$ \begin{aligned} \nabla^\perp S &= \langle \Lr, \nabla \Phi \rangle \\
\nabla^\perp \vec{R} &= \vec{L} \times \nabla^\perp \Phi +2H \nabla^\perp \Phi,  \end{aligned}$$
 satisfy: 
\begin{equation} 
\label{estimeersl2infini}
\begin{aligned}
 \|\nabla S \|_{L^{2,\infty} \left( \D_{r'} \right)} +\|\nabla \vec{R} \|_{L^{2,\infty} \left( \D_{r'} \right)}  &\le   \left\| \Lr e^{\lambda} \right\|_{L^{2,\infty} \left( \D_{r'} \right)}  + \left\| H \nabla \Phi \right\|_{L^2 \left( \D_{r'} \right) } \\ &\le \left(C_1+1 \right) \left\| H \nabla \Phi \right\|_{L^2 \left( \D \right) }.
\end{aligned}
\end{equation}
Noticing that $S$ and $\vec{R}$ are defined up to an additive constant, we can choose $S$ and $\vec{R}$ to be of null average value on $\D_{r'}$. 

The classic Poincaré–Wirtinger's inequality (see theorem 2, section 5.8.1 in \cite{evalawrpde}) yields for any $1<p<\infty$ and any $u$ such that $\nabla u \in L^p \left( \D_{r'} \right)$:
$$\left\| u - \bar{u} \right\|_{L^p \left( \D_{r'} \right) } \le C_{p ,r'}\left\| \nabla u \right\|_{L^p (\D_{r'})}$$
 with $C_{p,r'} \in \R_+$ and $\bar{u}$ the mean value of $u$ on $\D_{r'}$. 
 These inequalities can be extended using Marcinkiewitz interpolation theorem (see for example theorem 3.3.3 of  \cite{bibharmmaps}) to $L^{2, \infty}$: there exists $C_{r'}$ such that for any $u$ with $\nabla u \in L^{2,\infty} \left( \D \right)$
$$ \| u - \bar{u} \|_{L^{2,\infty} \left(\D_{r'} \right)} \le C_{r'} \| \nabla u \|_{L^{2,\infty} ( \D_{r'} )}.$$
Applied to $S$ and $\vec{R}$ (which are of null mean value), this yields:
$$ \| S \|_{W^{1, \left(2,\infty\right)} \left( \D_{r'} \right)} +\| \vec{R} \|_{ W^{1, \left(2,\infty \right)} \left( \D_{r'} \right)}  \le   C  \| H \nabla \Phi  \|_{L^2 \left( \D \right)},$$
where $C$ depends on $r'$. Since, thanks to (\ref{systemenRSPhiannex})
$$
\begin{aligned}
\Delta S &= \langle \nabla \vec{R} , \nabla^\perp \n \rangle, \\
\end{aligned}
$$
one can decompose $S = \sigma + s$ where  $s$ is harmonic and $\sigma$ is a solution of 
$$\left\{ \begin{aligned} &\Delta \sigma = \nabla \vec{R}. \nabla^\perp \n \text{ in }   \D_{r'} \\
&\sigma = 0 \text{ on } \partial  \D_{r'}. \end{aligned} \right.$$
 Using Wente's lemma (theorem \ref{wentel2infini}, in appendix) one finds:
\begin{equation} 
\label{equationensigma}
\begin{aligned}
\left\| \nabla \sigma \right\|_{L^2 \left(  \D_{r'} \right) } &\le C \| \nabla \vec{R} \|_{L^{2, \infty} \left(  \D_{r'} \right)} \| \nabla \n \|_{L^2 \left(  \D_{r'} \right) } \\
&\le  C \| H \nabla \Phi  \|_{L^2 \left( \D \right)},
\end{aligned}
\end{equation}
where $C$ depends on $C_0$ and $C_1$.
Meanwhile, Poisson's formula yields for $s$: 
\begin{equation} \label{equationensintermediaireb} \left\| \nabla s \right\|_{L^2 \left( \D_{\frac{r+r'}{2}} \right) } \le C \| S \|_{L^1 \left( \partial \D_{r'}  \right) }\end{equation}
where $C$ depends on $r$, and $r'$.

Using Marcinkiewitz interpolation theorem on trace operators yields
\begin{equation} \label{estimerSaubord} \| S \|_{L^1 \left( \partial \D_{r'}  \right) }\le C \left\| \nabla S \right\|_{L^{2,\infty} \left( \D_{r'}  \right)}\end{equation}
with $C$ depending on $r'$.
Combining  (\ref{estimeersl2infini}), (\ref{equationensintermediaireb}) and (\ref{estimerSaubord}) yields: 
\begin{equation}
\label{estimeeensl2}
\left\| \nabla s \right\|_{L^2 \left( \D_{\frac{r+r'}{2}} \right) } \le  C  \| H \nabla \Phi  \|_{L^2 \left( \D \right)},
\end{equation}
where $C$ depends on $r$, $r'$, $C_1$ and $C_0$.
Together (\ref{equationensigma}) and (\ref{estimeeensl2}) yield: 
$$\| \nabla S \|_{L^2  \left( \D_{\frac{r+r'}{2}} \right) } \le  C  \| H \nabla \Phi  \|_{L^2 \left( \D \right)}.$$
Working similarly on $ \vec{R}$ one finds 
\begin{equation}
\label{labonneestimeel2setr}
\| \nabla S \|_{L^2  \left( \D_{\frac{r+r'}{2}} \right) } + \| \nabla \vec{R} \|_{L^2  \left( \D_{\frac{r+r'}{2}} \right) }  \le   C \| H \nabla \Phi  \|_{L^2 \left( \D \right)}.
\end{equation}
This estimate can still be improved: let $S = \sigma' + s'$ with $s'$ harmonic and $\sigma'$
$$\left\{ \begin{aligned}& \Delta \sigma' = \nabla \vec{R}. \nabla^\perp \n \text{ in }   \D_{\frac{r+r'}{2}}\ \\
&\sigma' = 0 \text{ on } \partial  \D_{\frac{r+r'}{2}}.\end{aligned} \right.$$
 Using theorem \ref{wentel21} (in appendix) and (\ref{labonneestimeel2setr}) ensures
\begin{equation} 
\label{equationensigmaprime}
\begin{aligned}
\left\| \nabla \sigma' \right\|_{L^{2,1} \left(  \D_{\frac{r+r'}{2}} \right) } &\le C \| \nabla \vec{R} \|_{L^{2} \left(  \D_{\frac{r+r'}{2}} \right)} \| \nabla \n \|_{L^2 \left(  \D_{\frac{r+r'}{2}} \right) } \\
&\le  C  \| H \nabla \Phi  \|_{L^2 \left( \D \right)}.
\end{aligned}
\end{equation}
Using Poisson's formula allows one to control $s'$:
\begin{equation} \label{equationensintermediairebprime} \left\| \nabla s' \right\|_{L^{2,1} \left( \D_{\frac{3r+r'}{4}} \right) } \le C \| S \|_{L^1 \left( \partial \D_{\frac{r+r'}{2}}  \right) }.\end{equation}
As before, Marcinkiewitz interpolation on trace theorems yields
\begin{equation}
\label{estimeeensl2prime}
\left\| \nabla s' \right\|_{L^{2,1} \left( \D_{\frac{3r+r'}{4}} \right) } \le C  \| H \nabla \Phi  \|_{L^2 \left( \D \right)}.
\end{equation}
Together (\ref{equationensigmaprime}) and (\ref{estimeeensl2prime}) ensure
$$\| \nabla S \|_{L^{2,1}  \left( \D_{\frac{3r+r'}{4}} \right) } \le  C \| H \nabla \Phi  \|_{L^2 \left( \D \right)}.$$ 

Working analogously on $\vec{R}$ one finds 
\begin{equation}
\label{labonneestimeel21setr}
\| \nabla S \|_{L^{2,1}  \left( \D_{\frac{3r+r'}{4}} \right) } + \| \nabla \vec{R} \|_{L^{2,1}  \left( \D_{\frac{3r+r'}{4}} \right) }  \le  C \| H \nabla \Phi  \|_{L^2 \left( \D \right)}.
\end{equation}
Once more, $C$ depends on $r$, $r'$, $C_0$ and $C_1$ which concludes Step 1.
\newline
\underline{ \bf{ Step 2: }} {\bf $L^{2,1}$ control of $H \nabla \Phi$  } \newline
We simply use inequality (\ref{majorerhnablaphiparnablaR}), proven in appendix:
$$ \left| H \nabla \Phi \right| \le \frac{1}{2} \left|   \nabla \vec{R} \right|.$$ Combining it with (\ref{labonneestimeel21setr}) we find 
\begin{equation}
\label{labonneestimeel21hanablaphi}
\| H \nabla \Phi \|_{L^{2,1}  \left( \D_{\frac{3r+r'}{4}} \right) }  \le C  \| H \nabla \Phi  \|_{L^2 \left( \D \right)},
\end{equation}
which gives us the desired control on $H \nabla \Phi$.
\newline
\underline{ \bf{ Step 3: }} {\bf $L^{2,1}$ control of $\nabla \n$  } \newline
To expand these estimates to $\nabla \n$ we will use  equation (\ref{equationDeltanannexarticlemieux!}) (see appendix)
$$
\Delta \n + \nabla \n \times \nabla^\perp \n + 2 \mathrm{div} \left( H \nabla \Phi \right) = 0.
$$
Using corollary \ref{hodgedecompositionmoinsbon21} and (\ref{labonneestimeel21hanablaphi})  there exists  $\alpha \in W^{1,(2,1)}\left( \D_{\frac{3r+r'}{4}} \right)$ such that
\begin{equation}
\label{lealpha}
 \Delta \alpha= \mathrm{div} \left(H\nabla \Phi \right) 
\end{equation}
and
\begin{equation}
\label{lecontrolealpha}
\begin{aligned}
\left\| \alpha \right\|_{W^{1,(2,1)} \left( \D_{\frac{3r+r'}{4}} \right) }   &\le \| H \nabla \Phi \|_{L^{2,1}  \left( \D_{\frac{3r+r'}{4}} \right) } \\  &\le C  \| H \nabla \Phi  \|_{L^2 \left( \D \right)}.
\end{aligned}
\end{equation}
 Setting $\nu = \n -2 \alpha$ and using  (\ref{lecontrolealpha}) yields 
\begin{equation}
\label{lecontroledunuintermediaire}
\begin{aligned}
\left\| \nabla \nu \right\|_{L^2  \left( \D_{\frac{3r+r'}{4}} \right) } &\le \left\| \nabla \left( \n -2 \alpha \right) \right\|_{L^2  \left( \D_{\frac{3r+r'}{4}} \right) }   \\&\le \left\| \nabla  \n \right\|_{L^2  \left( \D_{\frac{3r+r'}{4}} \right) }   +2 \left\| \nabla  \alpha \right\|_{L^2  \left( \D_{\frac{3r+r'}{4}} \right) }   \\
&\le \left\| \nabla  \n \right\|_{L^2  \left( \D\right) }   +2 C\left\| \nabla  \alpha \right\|_{L^{2,1}   \left( \D_{\frac{3r+r'}{4}} \right) }\\
&\le  \left\| \nabla  \n \right\|_{L^2  \left( \D\right) }  +  C  \| H \nabla \Phi  \|_{L^2 \left( \D \right)}\\
&\le C \left\| \nabla  \n \right\|_{L^2  \left( \D\right) }.
\end{aligned}
\end{equation}

Besides, $\nu $ satisfies $$\Delta \nu  + \nabla \n \times \nabla^\perp \n=0.$$
We split $\nu = \nu_1 + \nu_2$ with $\nu_2$ harmonic and  $\nu_1$ solution of 
$$ \left\{ \begin{aligned}
&\Delta \nu_1  + \nabla \n \times \nabla^\perp \n = 0  \text{ in }  \D_{\frac{3r+r'}{4}} \\
&\nu_1 = 0 \text{ on } \partial  \D_{\frac{3r+r'}{4}}.
\end{aligned} \right.$$
Using theorem \ref{wentel21} we bound 
\begin{equation}
\label{onestimenu1}
\| \nabla \nu_1 \|_{L^{2,1} \left( \D_{\frac{3r+r'}{4}} \right)} \le C \| \nabla \n \|^2_{L^{2} \left( \D_{\frac{3r+r'}{4}} \right)}.
\end{equation}

Using the same method as for the estimates on $s'$ (see (\ref{equationensintermediairebprime}) - (\ref{estimeeensl2prime})) and applying (\ref{lecontroledunuintermediaire})  we find 
\begin{equation}
\label{onestimenu2}
\| \nabla \nu_2 \|_{L^{2,1} \left( \D_{r} \right)} \le C \|\nabla \nu \|_{L^{2} \left( \D_{\frac{3r+r'}{4}}. \right)} \le C \left\| \nabla  \n \right\|_{L^2  \left( \D\right) } .
\end{equation}
Combining (\ref{onestimenu1}) and (\ref{onestimenu2}) yields 
\begin{equation}
\label{leboncotrole21dunu}
\| \nabla \nu \|_{L^{2,1} \left( \D_r \right) } \le C \left\| \nabla \n \right\|_{L ^{2} \left( \D \right) }.
\end{equation}
Since $\n = \nu +2 \alpha$, (\ref{lecontrolealpha}) and  (\ref{leboncotrole21dunu}) ensure
$$\begin{aligned}
\| \nabla \n \|_{L^{2,1} \left( \D_r \right) } &\le \| \nabla \nu \|_{L^{2,1} \left( \D_r \right) } +2 \| \nabla \alpha  \|_{L^{2,1} \left( \D_r \right) } \\
&\le C \left\| \nabla \n \right\|_{L ^{2} \left( \D \right) },
\end{aligned}$$
which concludes the proof.
\end{proof}

Theorem \ref{n21controlintro} follows from combining theorems \ref{epsilonreularitylintro} and  \ref{n21controlsousLL2inftyintro}.

\section{$\varepsilon$-regularity results for weak Willmore immersions: proof of theorem \ref{epsilonregularitehnablaphifaibleintro}}
\label{lasection2}
We will first prove a more adaptable result:
\begin{theo}
\label{epsilonregularitehnablaphifaibleLL2inftyi}
Let $\Phi \in \mathcal{E}\left(\D \right)$ satisfy the hypotheses of theorem \ref{epsilonreularitylintro}.
We assume there exists $r'<1$ and $C_1>0$ such that $$ \left\| \Lr e^{\lambda} \right\|_{L^{2,\infty} \left( \D_{r'} \right)} \le C_1 \left\| H \nabla \Phi \right\|_{L^2 \left( \D \right) }$$
where $\Lr$ is given by (\ref{LRSintro}).
Then, there exists $\varepsilon_0'$  depending only on $C_0$ such that if $$  \left\| H \nabla \Phi \right\|  \le \varepsilon_0'$$ then for any $r< r'$ there exists a constant  $ C \in \R$  depending on  $r$, $C_0$, $p$ and $C_1$ such that
$$\left\| H \nabla \Phi \right\|_{L^\infty \left( \D_{ r  } \right)} \le   C  { \left\| H \nabla \Phi \right\|_{L^2 \left( \D \right)}},$$ 
and 
$$\left\| \nabla \Phi \right\|_{W^{2,p} \left( \D_{ r  } \right)} \le      C  \| \nabla \Phi \|_{L^2 \left( \D \right) } $$
for all $p<\infty$.
\end{theo}

\begin{proof}
Let $r<r'<1$, we follow the outline given in the introduction.
\newline
\underline{ \bf{ Step 1: }} {\bf $W^{1,(2,1)}$ control on the Willmore quantities  } \newline
Let $\Lr$  satisfy our hypothesis.
Theorem \ref{n21controlsousLL2inftyintro} gives:
\begin{equation}
\label{controlerspourlemmeepsilonreghfaible}
\left\| \nabla S  \right\|_{L^{2,1} \left( \D_{\frac{r+r'}{2}} \right) } + \left\| \nabla \vec{R}  \right\|_{L^{2,1} \left( \D_{\frac{r+r'}{2}} \right) } \le C_1{ \left\| H \nabla \Phi \right\|_{L^2 \left( \D \right)}}.
\end{equation}\newline
\underline{ \bf{ Step 2: }} {\bf $W^{1,q}$ control on the Willmore quantities, for $q>2$ }\newline
Thanks to (\ref{systemenRSPhiannex}) and (\ref{systemenRSPhiannexmieux}) we can decompose in any $B_t(p)$, with $p \in  \D_{\frac{r + r'}{2}}$ and $t$ sufficiently small, $S = \sigma + s $ and $ \vec{R} = \vec{\rho} + \vec{r}$, with
\begin{equation}
  \left\{ \label{lesigma} \begin{aligned}
\Delta \sigma &= \Delta S = \langle H \nabla \Phi , \nabla^\perp \vec{R} \rangle = - \langle \nabla \n , \nabla^\perp \vec{R} \rangle \text{ in }  B_t(p)\\
\sigma &= 0 \text{ on } \partial B_t (p),
\end{aligned}
\right.
\end{equation}
\begin{equation}
 \left\{\label{lerho} \begin{aligned}
\Delta \vec{\rho} &=\Delta \vec{R} = - H \nabla \Phi \times \nabla^\perp \vec{R} - \nabla^\perp S H\nabla \Phi  \\&=  \nabla\n \times \nabla^\perp \vec{R} + \nabla^\perp S \nabla \n  \text{ in }  B_t(p)\\
\vec{\rho} &= 0 \text{ on } \partial B_t (p)
\end{aligned}
\right.
\end{equation}
\begin{equation}
\left\{ \begin{aligned}
\Delta s &=0 \text{ in }   B_t (p)\\
s &= S \text{ on } \partial  B_t (p),
\end{aligned}
\right.
\end{equation}
\begin{equation}
 \left\{ \begin{aligned}
\Delta \vec{r} &=0 \text{ in }   B_t (p)\\
\vec{r} &= \vec{R} \text{ on } \partial  B_t (p).
\end{aligned}
\right.
\end{equation}

Since $s$ and $\vec{r}$ are harmonic functions, $ l \rightarrow \frac{1}{l^2} \int_{B_l(p)} \left| \nabla s \right|^2$ and  $ l \rightarrow \frac{1}{l^2} \int_{B_l(p)} \left| \nabla \vec{r} \right|^2$ are classically non-decreasing  (see lemma IV.1 in \cite{bibconservationlawsforconformallyinvariantproblems}). It follows that
\begin{equation}
\label{controlel2partieharmonique}
\begin{aligned}
 \|\nabla s \|^2_{L^2\left(  B_{\frac{t}{2}} (p) \right)}  &\le \frac{1}{4}   \|\nabla s \|^2_{L^2 \left( B_{t} (p) \right)}, \\
  \|\nabla \vec{r} \|^2_{L^2 \left( B_{\frac{t}{2}} (p) \right)}  &\le \frac{1}{4}   \|\nabla \vec{r}\|^2_{L^2 \left( B_{t} (p) \right) }.
\end{aligned}
\end{equation}
Furthermore, thanks to (\ref{lesigma})
and theorem \ref{wentel21} we have
\begin{equation}
\label{controlel21surlapartiepleines}
\| \nabla \sigma \|_{L^{2,1} \left( B_t(p) \right)} \le  C \| \nabla \vec{R} \|_{L^2 \left( B_t(p) \right) }\| \nabla \n \|_{L^2 \left( B_t(p) \right) }.
\end{equation} 
 Thanks to (\ref{lesigma})
and theorem \ref{calderonzygmundl1l2infini} we find 
\begin{equation}
\label{controlesigmal2infini}
\begin{aligned}
\left\| \nabla \sigma \right\|_{L^{2,\infty} \left(  B_t(p) \right) } &\le C \left\| \langle H \nabla \Phi , \nabla^\perp \vec{R} \rangle \right\|_{L^{1} \left(  B_t(p) \right) } \\ 
&\le  C \left\| \nabla \vec{R} \right\|_{L^{2} \left(  B_t(p) \right) } \left\| H \nabla \Phi \right\|_{L^{2} \left(  B_t(p) \right) }.
\end{aligned}
\end{equation}
Exploiting the duality of $L^{2,1}$ and $L^{2, \infty}$ , (\ref{controlel21surlapartiepleines}) and (\ref{controlesigmal2infini}) yield
\begin{equation}
\label{controlel2sigma}
\begin{aligned}
\left\| \nabla \sigma \right\|^2_{L^{2} \left(  B_t(p) \right) } &\le \left\| \nabla \sigma \right\|_{L^{2,\infty} \left(  B_t(p) \right) }\| \nabla \sigma \|_{L^{2,1} \left( B_t(p) \right)} \\
&\le C \left( \left\| \nabla \n \right\|_{L^2 \left( \D \right)} \right) \| \nabla \vec{R} \|^2_{L^2 \left( B_t(p) \right) } \left\| H \nabla \Phi \right\|_{L^{2} \left( \D \right) }.
\end{aligned}
\end{equation}
Working similarly with $\vec{\rho}$ we find 
\begin{equation}
\label{controlel2rho}
\left\| \nabla \vec{\rho} \right\|^2_{L^{2} \left(  B_t(p) \right) } \le  C \left( \left\| \nabla \n \right\|_{L^2 \left( \D \right)} \right) \left( \| \nabla \vec{R} \|^2_{L^2 \left( B_t(p) \right) } + \| \nabla S \|^2_{L^2 \left( B_t(p) \right) }  \right) \left\| H \nabla \Phi \right\|_{L^{2} \left(  \D \right) }.
\end{equation}
We remind the reader that the constant from theorems  \ref{wentel21} and \ref{calderonzygmundl1l2infini} are universal due to the scale invariance properties of the $L^2$, $L^{2, \infty}$ and $L^{2,1}$ norms. The constants in (\ref{controlel2sigma}) and (\ref{controlel2rho}) then do depend solely on $\left\| \nabla \n \right\|_{L^2 (\D ) }$.

We can combine (\ref{controlel2partieharmonique}), (\ref{controlel2sigma}) and  (\ref{controlel2rho}) to get 
\begin{equation}
\label{diminuerlerayonsurunepetiteboule}
\begin{aligned}
\left\|  \nabla S \right\|^2_{L^2 \left( B_{\frac{t}{2} } (p) \right)}  + \|  \nabla \vec{R} \|^2_{L^2 \left( B_{\frac{t}{2} } (p) \right)}  &\le \frac{1}{2} \left( \left\|  \nabla s \right\|^2_{L^2 \left( B_{t } (p) \right)}  +  \| \nabla \vec{r}  \|^2_{L^2 \left( B_{t } (p) \right)} \right) \\&+2C \left( \left\| \nabla \n \right\|_{L^2 \left( \D \right)} \right) \left( \| \nabla \vec{R} \|^2_{L^2 \left( B_t(p) \right) } + \| \nabla S \|^2_{L^2 \left( B_t(p) \right) }  \right) \left\| H \nabla \Phi \right\|_{L^{2} \left(  \D \right) }
\\
&\le \left( \frac{1}{2} +  \left\| H \nabla \Phi \right\|_{L^{2} \left( \D \right) }  C \right) 
\left( \left\|  \nabla S\right\|^2_{L^2 \left( B_{t } (p) \right)}  +  \| \nabla \vec{R}  \|^2_{L^2 \left( B_{t } (p) \right)} \right),
\end{aligned}
\end{equation}
where $C$ depends solely on $\left\| \nabla \n \right\|_{L^2\left( \D \right)}$. Should $ \left\| H \nabla \Phi \right\|_{L^{2} \left( \D \right) }$ be small enough 
then (\ref{diminuerlerayonsurunepetiteboule}) would yield 
\begin{equation}
\label{diminuerlerayonsurunepetiteboulequantifie}
\begin{aligned}
\left\|  \nabla S \right\|^2_{L^2 \left( B_{\frac{t}{2} } (p) \right)}  + \|  \nabla \vec{R} \|^2_{L^2 \left( B_{\frac{t}{2} } (p) \right)}&\le \frac{3}{4} \left( \left\|  \nabla S\right\|^2_{L^2 \left( B_{t } (p) \right)}  +  \| \nabla \vec{R}  \|^2_{L^2 \left( B_{t } (p) \right)} \right).
\end{aligned}
\end{equation}

Since the chosen $\varepsilon_0'$ depends only of $\left\| \nabla \n \right\|_{L^2 \left( \D \right)}$, (\ref{diminuerlerayonsurunepetiteboulequantifie}) is uniformly true for all $B_l(p) \subset \D_{\frac{2r+r'}{3}}$ and yields a Morrey-type estimate on $\D_{\frac{2r+r'}{3}}$. Through usual estimates on Riesz potentials, see for instance theorem 3.1 in \cite{bibrieszpot} , it entails

\begin{equation}
\exists q>2 \text{ s.t. }\quad \left\|  \nabla S\right\|_{L^q \left( \D_{\frac{3r+r'}{4}} \right) }  +   \left\|  \nabla \vec{R} \right\|_{L^q \left( \D_{\frac{3r+r'}{4}} \right) }   \le C_q  \left( \left\|  \nabla S\right\|_{L^2 \left( \D_{\frac{r+r'}{2}} \right) }  +  \left\| \nabla \vec{R}  \right\|_{L^2 \left( \D_{\frac{r+r'}{2}} \right) } \right) .
\end{equation}
\newline
\underline{ \bf{ Step 3: }} {\bf $L^\infty$ control on $H \nabla \Phi$ } \newline
Thanks to Step 2 and (\ref{majorerhnablaphiparnablaR}) we deduce 
$$\left\| H \nabla \Phi \right\|_{L^q \left( \D_{\frac{3r+r'}{4}} \right) } \le  C_q  \left( \| \nabla S \|_{L^2 \left( \D_{\frac{r+r'}{2}} \right) } + \| \nabla \vec{R} \|_{L^2 \left( \D_{\frac{r+r'}{2}} \right) } \right).$$
The criticality of system (\ref{systemenRSPhiannexmieux}) is thus broken: $\Delta S, \Delta \vec{R}$ are in $L^{\frac{q}{2}}$ with $\frac{q}{2}>1$. One can apply classic Calder\'on-Zygmund theory (see for instance theorem 9.9 and 9.11 of \cite{bibellipticpartialdifferentialequations}) to start a bootstrap of limiting regularity $L^\infty$ on $ H \nabla \Phi $. \emph{In fine}, one has with estimate (\ref{controlerspourlemmeepsilonreghfaible})

\begin{equation}
\label{regularitemaxpetith}
 \left\| \nabla S \right\|_{W^{1,p} \left( \D_{\frac{4r+r'}{5}} \right) } + \left\|  \nabla \vec{R}  \right\|_{W^{1,p} \left( \D_{\frac{4r+r'}{5}} \right) } + \| H\nabla \Phi \|_{L^\infty\left( \D_{\frac{4r+r'}{5}} \right) }   \le      C { \left\| H \nabla \Phi \right\|_{L^2 \left( \D \right)}}
\end{equation}
for all $p< \infty$. Here $C$ is a real constant which depends on $r$, $r'$, $C_0$ and $C_1$.
\newline
\underline{ \bf{ Step 4: }} {\bf $W^{3,p}$ control on $\Phi$ }\newline
The control on $\nabla \Phi$ is obtained by a similar  Calder\'on-Zygmund bootstrap on equation
$$ 2 \Delta \Phi = \nabla^\perp S \nabla \Phi + \nabla^\perp \vec{R} \times \nabla \Phi,$$
which achieves the proof.
\end{proof}

One only needs to combine theorems \ref{epsilonreularitylintro} and \ref{epsilonregularitehnablaphifaibleLL2inftyi} to prove theorem \ref{epsilonregularitehnablaphifaibleintro}.

Theorems \ref{epsilonregularitehnablaphifaibleintro} and \ref{epsilonregularitehnablaphifaibleLL2inftyi} can be applied on disks of arbitrari radii, at the cost of a control depending on the radius of the disk. Indeed a rescaling, very similar to what has already been done in the proof of corollary \ref{controlconformalfactorcor} yields the following theorem.

\begin{theo}
\label{epsilonregularitehnablaphifaibleLL2inftyscale}
Let $\Phi \in \mathcal{E}\left(\D_\rho \right)$  satisfy the hypotheses of theorem \ref{epsilonreularitylintro}.
We assume there exists $r'<1$ and $C_1>0$ such that $$ \left\| \Lr e^{\lambda} \right\|_{L^{2,\infty} \left( \D_{r'\rho} \right)} \le C_1 \left\| H \nabla \Phi \right\|_{L^2 \left( \D_\rho \right) }.$$
Then there exists $\varepsilon_0'$  depending only on $C_0$ such that if $$  \left\| H \nabla \Phi \right\|_{L^2 \left( \D_\rho \right) }  \le \varepsilon_0',$$ then for any $r< r'$ there exists a constant  $ C \in \R$  depending on  $r$, $C_0$, $p$ and $C_1$ such that
$$\left\| H \nabla \Phi \right\|_{L^\infty \left( \D_{ r \rho } \right)} \le   \frac{C}{\rho} { \left\| H \nabla \Phi \right\|_{L^2 \left( \D_\rho \right)}}.$$ 
\end{theo}

\section{Control of $\Lr e^\lambda$ on an annulus}
\label{lasection3}
In this section, we focus on a control of $\Lr$ on annuli of small energy, independantly of its conformal class (see (VI.23) in \cite{bibenergyquant}).
\begin{theo}
\label{ll2infinianneau}
Let $R>0$ and $\Phi \in \mathcal{E}\left(\D_R \right)$ be a conformal weak Willmore immersion.  Let  $\n$ denote its Gauss map, $H$ its mean curvature and $\lambda$ its conformal factor.
We assume $$ \left\| \nabla \n \right\|_{L^2\left(\D_R \right) } +\left\| \nabla \lambda \right\|_{L^{2,\infty} \left( \D_R \right)} \le C_0.$$
Then there exists $\varepsilon_0>0$ (independant of $\Phi$) such that if  $0 < 8r < R$ and 
$$ \sup_{r<s<\frac{R}{2} } \int_{\D_{2s} \backslash \D_s } \left| \nabla \n \right|^2 \le \varepsilon_0,$$
 then  there exists $\vec{\mathcal{L}} \in \R^3$ and $C \in \R$ depending on $C_0$ but not on the conformal class of $\D_R \backslash \D_r$ such that 
$$\left\| e^{\lambda} \left( \Lr -\vec{ \mathcal{L}} \right) \right\|_{L^{2,\infty} \left( \D_{\frac{R}{2}} \backslash \D_{2r} \right)} \le C { \left\| H \nabla \Phi \right\|_{L^2 \left( \D_R \right)}},$$ 
where $\Lr$ is given by (\ref{LRSintro}).
\end{theo}
Once more, we will follow Y. Bernard and T. Rivière's proof, with a few tweaks in order to obtain a control of $\Lr e^\lambda$ by $H\nabla \Phi$ instead of $\nabla \n$. It is important for $\Phi$ to be well-defined, and the bound on its conformal factor and Gauss map to stand, on the whole disk and not merely on the annulus. We refer the reader to \cite{MR3843372} for a study of what can happen otherwise.  In the context of  theorem \ref{energyquandberriv}, theorem \ref{ll2infinianneau} gives controls on the neck regions around the concentration points.
\begin{proof}
\underline{ \bf{ Step 1: }} {\bf Pointwise estimates  on $\vec{H}$ and $\nabla \vec{H}$}\newline
We set ourselves in the setting of theorem \ref{ll2infinianneau} and consider $\Phi \in \mathcal{E}\left(\D_R \right)$ a conformal weak Willmore immersions of Gauss map $\n$, mean curvature $H$, conformal factor $\lambda$ and tracefree second fundamental form $\Ar$.
We assume that $$\left\| \nabla \n \right\|_{L^2\left(\D_R \right) }+\| \nabla \lambda \|_{L^{2,\infty} \left( \D_R \right)} \le C_0< \infty,$$ and that 
\begin{equation}
\label{lhypothesesurlenpetit}
\sup_{r<s<\frac{R}{2} } \int_{\D_{2s} \backslash \D_s} |\nabla \n |^2 \le \varepsilon_0.
\end{equation}

Consider $x \in \D_{\frac{R}{2}} \backslash \D_{2r}$, then $B_{\frac{|x|}{4}} (x) \subset \D_{2|x|} \backslash \D_{\frac{|x|}{2} }$ and thus (\ref{lhypothesesurlenpetit}) implies \begin{equation} \label{onestpetitx}\int_{B_{\frac{|x|}{4}} (x)} |\nabla \n |^2 \le \varepsilon_0.\end{equation}
On $B_{\frac{|x|}{4}} (x)$ one can then apply either theorem \ref{epsilonregularityRivierecontrolphi}, or theorem \ref{epsilonregularitehnablaphifaibleintro} (with $r_0= 1$ since (\ref{onestpetitx}) stands) to deduce 
\begin{equation}
\label{oncontrolelenpointwise}
\left\| \nabla \n \right\|_{L^\infty\left(B_{\frac{|x|}{8}} (x)\right)} \le \frac{C}{|x|} \left\| \nabla \n \right\|_{L^2\left(B_{\frac{|x|}{4}} (x)\right)},
\end{equation}
and 
\begin{equation}
\label{oncontroleleHpointwise}
\left\| H \nabla \Phi \right\|_{L^\infty\left(B_{\frac{|x|}{8}} (x)\right)} \le \frac{C}{|x|} \left\|H \nabla \Phi \right\|_{L^2\left(B_{\frac{|x|}{4}} (x)\right)}.
\end{equation}
Here $C$ depends on $C_0$. Corollary \ref{controlconformalfactorcor} then ensures a Harnack inequality on $B_{\frac{|x|}{8}} (x)$, meaning there exists $\Lambda \in \R$ and $C$ depending only on $C_0$ such that for all $p \in B_{\frac{|x|}{8}} (x)$ we have
\begin{equation}\label{leharnacksurlabouleenx}\frac{e^\Lambda}{C} \le e^{\lambda (p)} \le Ce^\Lambda.\end{equation} This allows one to control $H$ with (\ref{oncontroleleHpointwise}):
\begin{equation}
\label{oncontroleleHlevraipointwise}
\left\| H  \right\|_{L^\infty\left(B_{\frac{|x|}{8}} (x)\right)} \le \frac{Ce^{-\Lambda}}{|x|} \left\|H \nabla \Phi \right\|_{L^2\left(B_{\frac{|x|}{4}} (x)\right)}.
\end{equation}
Since $\Phi$ is Willmore, it satisfies (\ref{lequationdewillmoreclassique}):
$$\Delta H + \big| \Ar \big|^2 H = 0.$$
Combining  (\ref{oncontrolelenpointwise}), (\ref{oncontroleleHlevraipointwise}) and (\ref{Arborneparnablanannex}) yields
$$\left\| \big| \Ar \big|^2 H \right\|_{L^\infty\left(B_{\frac{|x|}{8}} (x)\right)} \le \frac{Ce^{-\Lambda}}{|x|^3} \left\|H \nabla \Phi \right\|_{L^2\left(B_{\frac{|x|}{4}} (x)\right)}.$$
Then $$\left\| \Delta H \right\|_{L^\infty\left(B_{\frac{|x|}{8}} (x)\right)} \le \frac{Ce^{-\Lambda}}{|x|^3} \left\|H \nabla \Phi \right\|_{L^2\left(B_{\frac{|x|}{4}} (x)\right)}.$$
Classic Calder\'on-Zygmund results (see for instance theorem 9.9 and 9.11 of \cite{bibellipticpartialdifferentialequations}) ensure that 
\begin{equation} \label{oncontrolelenablaHparHetcestbien}\left\| \nabla H \right\|_{L^\infty\left(B_{\frac{|x|}{16}} (x)\right)} \le \frac{Ce^{-\Lambda}}{|x|^2} \left\|H \nabla \Phi \right\|_{L^2\left(B_{\frac{|x|}{4}} (x)\right)}.\end{equation}
Combining first (\ref{oncontroleleHpointwise}) and (\ref{leharnacksurlabouleenx}), and then  (\ref{oncontrolelenablaHparHetcestbien}) and (\ref{leharnacksurlabouleenx}) yields when evaluated at $x$
\begin{equation}
\label{labonneestimeepointwiseH}
e^{\lambda(x)} \left| H(x) \right| \le C \delta(|x|),
\end{equation}
\begin{equation}
\label{labonneestimeepointwisenablaH}
e^{\lambda(x)} \left| \nabla H(x) \right| \le \frac{C}{|x|} \delta(|x|),
\end{equation}
where $$\delta(s) = \frac{1}{s} \left\| H \nabla \Phi \right\|_{L^2 \left(\D_{2s} \backslash \D_{\frac{s}{2}}  \right)}.$$
Since $\nabla \vec{H} = \nabla H \n + H \nabla \n$, we can extend (\ref{labonneestimeepointwiseH}) and (\ref{labonneestimeepointwisenablaH}) to $\vec{H}$ and $\nabla \vec{H}$ thanks to (\ref{oncontrolelenpointwise}), which yields the desired estimates.
\newline
\underline{ \bf{ Step 2: }} {\bf Controls on $\delta$}\newline
Clearly we have \begin{equation} \label{lestimeeendelta} s \delta(s) \le \left\| H \nabla \Phi \right\|_{L^2 \left( \D_R \backslash \D_{\frac{r}{2}} \right)}. \end{equation}
Further, for any  positive function $f$:
\begin{equation} \label{lecontroleaveclepetitlogdedans}\begin{aligned} \int_r^{\frac{R}{2}} \frac{1}{s} \int_{\frac{s}{2}}^{2s} f(t)dt ds &\le \int_{\frac{r}{2}}^R \int_{\frac{t}{2}}^{2t} \frac{1}{s} f(t) ds dt \\
&\le \int_{\frac{r}{2}}^R f(t) \log \left( \frac{2t}{\frac{t}{2}} \right) dt \\
&\le \log  4 \int_{\frac{r}{2}}^R f(t)dt. \end{aligned}\end{equation}
Applying (\ref{lecontroleaveclepetitlogdedans}) with $f(t)=\int_{\partial \D_t} \left| H \nabla \Phi\right|^2d\sigma_{\partial \D_t}$, we find 
\begin{equation}
\label{uneautreestimeesurdelta}
\int_r^{\frac{R}{2}} s \delta^2(s)ds \le \log 4 \left\| H \nabla \Phi \right\|^2_{L^2\left( \D_R \backslash \D_{\frac{r}{2}}\right)},
\end{equation}
while  with $\tilde f(t)=\int_{\partial \D_t} \left| \nabla \n\right|^2d\sigma_{\partial \D_t}$, this yields  (VI.9) in \cite{bibenergyquant}: 
\begin{equation}
\label{uneautreestimeesurdeltatilde}
\int_r^{\frac{R}{2}} s \tilde \delta^2(s)ds \le \log 4 \left\| \nabla \n \right\|^2_{L^2\left( \D_R \backslash \D_{\frac{r}{2}}\right)},
\end{equation}
 where$$\tilde \delta (s) =  \frac{1}{s} \left\| \nabla \n \right\|_{L^2 \left(\D_{2s} \backslash \D_{\frac{s}{2}}  \right)}.$$  
\newline
\underline{ \bf{ Step 2: }} {\bf Exploitation and control of $\Lr$ }\newline
Let $\Lr$ be a first Willmore quantity of $\Phi$ on $\D_R$, i.e. satisfying (\ref{definitionl}). From (\ref{definitionl}), (\ref{oncontrolelenpointwise}), (\ref{labonneestimeepointwiseH}) and (\ref{labonneestimeepointwisenablaH}) we deduce for all $x \in \D_{\frac{R}{2}} \backslash \D_{2r}$
\begin{equation}
\label{estimeepointwiseonnablaL}
\left|\nabla \Lr\right|(x) \le \frac{C e^{-\lambda(x)}}{|x|} \delta( |x|).
\end{equation}
We consider for any $r\le t \le R$
$$\Lr_t:= \frac{1}{\left|\partial \D_t \right|} \int_{\partial \D_t} \Lr d\sigma_{\partial \D_t}.$$
Then given $x \in \D_{\frac{R}{2} } \backslash \D_{2r}$
\begin{equation} \label{eqauxint}\begin{aligned}\left| \Lr(x) - \Lr_{|x|} \right| &\le \int_{\partial \D_{|x|}} \left|\nabla \Lr \right| d\sigma_{\partial \D_{|x|} } \\
&\le \int_{\partial \D_{|x|}} \frac{C e^{-\lambda(x)}}{|x|} \delta( |x|)d\sigma_{\partial \D_{|x|} } \\
&\le C \delta(|x|) \int_0^{2\pi}  e^{-\lambda(|x|e^{i\theta})}d\theta. \end{aligned}\end{equation}
One key step of the proof is controlling the conformal factor with a Harnack inequality. However as the conformal class of the annulus degenerates, the number of small energy disks needed to cover it goes to infinity. Thus, we have no hope to properly estimate the conformal factor by a constant. Y. Bernard and T. Rivière have however shown that a function of type $r^d$ can be a good approximation, as stated in lemma V.3 of \cite{bibenergyquant} (see below).
\begin{lem}
There exists a constant $\eta>0$ with the following property. Let $0<4r<R< \infty$. If $\Phi$ is any (weak) conformal immersion of $\Omega:= \D_R \backslash \D_r$ into $\R^3$ with $L^2$-bounded second fundamental form and satisfying
$$\left\| \nabla \n \right\|_{L^{2,\infty} \left(\Omega \right)} < \sqrt{\eta},$$
then there exist $\frac{1}{2} < \alpha < 1$ and $A \in \R$ depending on $R$, $r$, $m$ and $\Phi$ such that 
\begin{equation}
\label{leharnackdanslecou}
\left\| \lambda (x) - d \log |x| - A \right\|_{L^\infty \left( \D_{\alpha R} \backslash \D_{\frac{r}{\alpha}} \right)} \le C \left( \left\| \nabla \lambda \right\|_{L^{2,\infty} \left( \Omega \right)} + \int_\Omega \left| \nabla \n \right|^2 \right),
\end{equation} 
where $d$ satisfies 
\begin{equation}
\begin{aligned}
\left| 2 \pi d - \int_{\partial \D_r } \partial_r \lambda dl_{\partial \D_r} \right| \le& C \left[ \int_{\D_{2r} \backslash \D_r } \left| \nabla \n \right|^2 \right.\\
&\left. + \frac{1}{\log \frac{R}{r}} \left( \left\| \nabla \lambda \right\|_{L^{2,\infty} \left(\Omega \right)} + \int_{\Omega} \left| \nabla \n \right|^2 \right)\right].
\end{aligned}
\end{equation}
\end{lem}
In our case, (\ref{leharnackdanslecou}) implies the following Harnack inequality for all $x \in \D_{\frac{R}{2} } \backslash \D_{2r}$
\begin{equation}
\label{levraiharnackcou}
\frac{e^A |x|^d}{C} \le e^{\lambda(x)} \le Ce^A |x|^d,
\end{equation}
with $d$, $A$ in $\R$, and $C$ a constant depending on $C_0$. Then (\ref{eqauxint}) yields
\begin{equation} \label{equationauxxint2} \begin{aligned}\left| \Lr(x) - \Lr_{|x|} \right|
&\le C \delta(|x|)   e^{-\lambda(x)}, \end{aligned}\end{equation}
with $C$ depending on $C_0$. We can then estimate $\Lr - \Lr_{|x|}$  with (\ref{uneautreestimeesurdelta}) and (\ref{equationauxxint2}):
\begin{equation}
\label{equationauxint3}
\int_{\D_{\frac{R}{2}} \backslash \D_{2r}} e^{2\lambda} \left| \Lr - \Lr_{|x|} \right|^2 dx \le C \int_{2r}^{\frac{R}{2}} r \delta^2(r)dr \le C \left\| H \nabla \Phi \right\|^2_{L^2 \left( \D_R \backslash \D_{\frac{r}{2}}\right) }.
\end{equation}

We will control similarly $\frac{d \Lr_t}{dt} = \frac{1}{2\pi} \int_0^{2\pi} \frac{\partial \Lr}{\partial t } (t, \theta) d\theta$.
We use expression (\ref{eqauxsurleL}) of $\nabla \Lr$  and deduce 
$$\frac{1}{2\pi} \int_0^{2\pi} \frac{\partial \Lr}{\partial t } (t, \theta) d\theta = \frac{3}{2\pi} \int_0^{2\pi} H \partial_\theta \n d\theta + \frac{1}{2\pi} \int_0^{2\pi} \partial_\nu \n \times \vec{H} d \theta.$$
Using (\ref{oncontrolelenpointwise}), (\ref{labonneestimeepointwiseH}) and  (\ref{levraiharnackcou}) we deduce from this
\begin{equation}
\label{equauxint4}
\left| \frac{d \Lr_t}{dt}  \right| \le C e^{-A} \frac{\delta(t) \tilde \delta (t)}{t^d}.
\end{equation}
Defining $a(t) = \left| \Lr_t \right|$ yields $\left|\frac{d a }{dt} \right| \le \left|  \frac{d \Lr_t}{dt} \right|$ which, combined with (\ref{levraiharnackcou}) and (\ref{equauxint4}) ensures 
\begin{equation}
\label{equauxint5}
\left|\frac{d a }{dt} \right| \le C e^{-A} \frac{\delta(t) \tilde \delta (t)}{t^{d}}.
\end{equation}
Then $$\begin{aligned}\int_{2r}^{\frac{R}{2}} s^{1+d} \left|\frac{d a }{ds} \right|(s)ds &\le Ce^{-A} \int_{2r}^{\frac{R}{2}} s \delta(s) \tilde \delta(s) \\
&\le Ce^{-A} \left( \int_{2r}^{\frac{R}{2}} s \delta(s)^2 ds \right)^{\frac{1}{2}} \left( \int_{2r}^{\frac{R}{2}} s {\tilde \delta}(s)^2ds \right)^{\frac{1}{2}}.\end{aligned}$$
 We can thus apply (\ref{uneautreestimeesurdelta}) and (\ref{uneautreestimeesurdeltatilde}) and conclude 
\begin{equation}
\label{estimeeglobauxre}
\begin{aligned}
\int_{2r}^{\frac{R}{2}} s^{1+d} \left|\frac{d a }{ds} \right|(s)ds &\le  Ce^{-A} \left\| \nabla \n \right\|_{L^2 \left(\D_R \backslash \D_r \right)}  \left\| H \nabla \Phi \right\|_{L^2 \left(\D_R \backslash \D_r \right)}\\&  \le C C_0 e^{-A} \left\| H \nabla \Phi \right\|_{L^2 \left(\D_R \backslash \D_r \right)}.
\end{aligned}
\end{equation}
An integration by parts gives for any $r < \tau < T < R$,
$$\int_\tau^T s^{1+d} \frac{d a }{ds}(s) ds = T^{1+d} a (T) - \tau^{1+d} a(\tau) - (1+d) \int_\tau^T s^d a(s)ds.$$
Hence, since $a\ge 0$, we  have 
\begin{itemize}
\item
if $d \le -1$, for all $2r < t < \frac{R}{2}$, $$t^{1+d} a(t) \le (2r)^{1+d} a(2r) + \int_{2r}^{\frac{R}{2}} s^{1+d} \left|\frac{da}{ds}\right|(s) ds,$$
\item

if $d \ge -1$, for all $2r < t < \frac{R}{2}$, $$t^{1+d} a(t) \le \left(\frac{R}{2} \right)^{1+d} a \left(\frac{R}{2}\right) + \int_{2r}^{\frac{R}{2}} s^{1+d} \left|\frac{da}{ds}\right|(s) ds.$$
\end{itemize}
Then if $d\le -1$ we take $\int_{\partial \D_{2r}} \Lr = 0$ whereas if $d \ge -1$, we take $\int_{\partial \D_{\frac{R}{2}}} \Lr =0$.

In both cases for all $2r < |x| < \frac{R}{2}$, thanks to (\ref{estimeeglobauxre}), we have
\begin{equation}
\label{almostthereestimate}
\begin{aligned}
|x| e^{\lambda(x)} \left| \Lr_{|x|} \right| &\le |x|^{d+1} e^{A} a(|x|) \\
&\le e^{A} \int_{2r}^{\frac{R}{2}} s^{1+d} \left|\frac{da}{dt}\right|(s) ds \\
&\le C \left\| H \nabla \Phi \right\|_{L^2 \left(\D_R \backslash \D_r \right)},
\end{aligned}
\end{equation}
 where $C$ depends only on $C_0$.
Since $\frac{1}{|x|}$ is in $L^{2, \infty}$, we conclude with
\begin{equation}
\label{alllllmost}
\left\| e^{\lambda(x)}  \Lr_{|x|} \right\|_{L^{2,\infty} \left(\D_{\frac{R}{2}} \backslash \D_{2r} \right)} \le  C \left\| H \nabla \Phi \right\|_{L^2 \left(\D_R \backslash \D_r \right)}.
\end{equation}
Combined with (\ref{equationauxint3}), this yields the desired result:
$$\left\| e^{\lambda}  \Lr \right\|_{L^{2,\infty} \left(\D_{\frac{R}{2}} \backslash \D_{2r} \right)} \le  C \left\| H \nabla \Phi \right\|_{L^2 \left(\D_R \backslash \D_r \right)}.$$

The constant appearing in the theorem corresponds to the choice of $\int_{\partial \D_{2r}} \Lr = 0$ or $\int_{\partial \D_{\frac{R}{2}}} \Lr =0$ depending on $d$.
\end{proof}

\section{Simple minimal bubbling: proof of theorem \ref{convminibubbleintro}}
\label{lasection4}
In the following, $\tilde \Phi^\varepsilon:= \frac{\Phi^\varepsilon \left( \varepsilon . \right) - \Phi^\varepsilon (0) }{C^\epsilon} \,:\, \D_{\frac{1}{\varepsilon} } \rightarrow \R^3$ and $\tilde \n^\varepsilon$, $\tilde H^\varepsilon$ $ \tilde \lambda^\varepsilon$ will denote respectively its Gauss map, its mean curvature and its conformal factor. 
We can check: 
\begin{equation}
\label{rescalingdelagaussmap}
\tilde \n^\varepsilon = \n^\varepsilon \left( \varepsilon . \right),
\end{equation}
\begin{equation}
\label{rescalingduHnablaPhi}
\tilde H^\varepsilon \nabla \tilde \Phi^\varepsilon = \varepsilon   H^\varepsilon \nabla \Phi^\varepsilon  \left( \varepsilon . \right).
\end{equation}
Then for all $\frac{1}{\varepsilon} >R >0$ 
\begin{equation}
\label{integraledunablantilde}
 \int_{\D_{\epsilon R} } \left| \nabla  \n^\varepsilon \right|^2 dz = \int_{\D_R} \left| \nabla \tilde \n^\varepsilon \right|^2 dz,
\end{equation}
and
\begin{equation}
\label{integraleduHnablaPhi}
 \int_{\D_{\varepsilon R} } \left| H^\epsilon \nabla \Phi^\varepsilon \right|^2 dz = \int_{\D_R} \left| \tilde H^\varepsilon \nabla \tilde \Phi^\varepsilon \right|^2 dz.
\end{equation}
Hypothesis \ref{hypothesis5} implies  
\begin{equation}
\label{convergenceenergiendelabulle}
 \lim_{\epsilon \rightarrow 0} \int_{\D_{\varepsilon R} } \left| \nabla \n^\varepsilon \right|^2 dz = \int_{ \D_{ R} } \left| \nabla \n_\Psi \right|^2 dz ,
\end{equation}
\begin{equation}
\label{convergenceenergiedelabulle}
 \lim_{\epsilon \rightarrow 0} \int_{\D_{\varepsilon R} } \left| H^\varepsilon \nabla \Phi^\varepsilon \right|^2 dz = \int_{ \D_{ R} } \left| H_\Psi \nabla \Psi\right|^2 dz = 0.
\end{equation}
Besides, combining (\ref{HnablaPhiborneparnablanannex}) and hypothesis \ref{hypothesis3} yields
  \begin{equation}
\label{lenoneckenergysurleHnablaPhiarticle}
\lim_{R \rightarrow \infty}  \left( \lim_{\varepsilon \rightarrow 0} \int_{\D_{\frac{1}{R}} \backslash \D_{\varepsilon R} } \left| H^\varepsilon \nabla \Phi^\varepsilon \right|^2 dz \right) \le \lim_{R \rightarrow \infty}  \left( \lim_{\varepsilon \rightarrow 0} \int_{\D_{\frac{1}{R}} \backslash \D_{\varepsilon R} } \left| \nabla \n^\varepsilon \right|^2 dz \right) = 0.
\end{equation}
Together (\ref{convergenceenergiedelabulle}) and (\ref{lenoneckenergysurleHnablaPhiarticle}) ensure that 
for $R$ sufficiently big and $\varepsilon$  sufficiently small \begin{equation} \label{l'hypotheseduhpetit} \left\| H^\varepsilon \nabla \Phi^\varepsilon \right\|_{L^2 \left( \D_{\frac{1}{R} } \right)} \le \varepsilon_0' \left( M \right), \end{equation} with $\varepsilon_0'(M)$ given by theorem \ref{epsilonregularitehnablaphifaibleLL2inftyi}. 
Up to a rescaling, and thus without loss of generality we can assume that (\ref{l'hypotheseduhpetit}) stands on $\D$.
We will find a uniform $L^{2,\infty}$ bound  on a  first Willmore quantity, 
 theorem \ref{epsilonregularitehnablaphifaibleLL2inftyi} then gives the uniform controls proving theorem \ref{convminibubbleintro}.

Recalling (\ref{integraledunablantilde}) yields 
$$ \lim_{\varepsilon \rightarrow 0} \int_{\D_{\varepsilon R} } \left| \nabla \n^\varepsilon \right|^2 dz = \int_{ \D_{ R} } \left| \nabla \n_\Psi \right|^2 dz .$$
Then either $\Psi$ parametrizes a plane, and classical $\varepsilon$-regularity results yield smooth convergence (and there is de facto no real bubbling) or
for $R$ big enough, $$ \lim_{\varepsilon \rightarrow 0} \int_{\D_{\varepsilon R} } \left| \nabla \n^\varepsilon \right|^2 dz  > \frac{8\pi}{3}.$$ 
Then $$ \inf \left\{ s \left| \int_{B_s(p)} \left| \, \nabla \n^\varepsilon \right|^2 = \frac{8 \pi}{6}, \: \forall p \in \D \text{ s.t. } B_s(p) \subset \D \right. \right\} \rightarrow 0.$$ This means that the estimates given by theorem \ref{epsilonreularitylintro} degenerates as $\varepsilon$ goes to $0$. Finding a uniform control on $\Lr e^{\lambda}$ will require a "bubble-neck" decomposition.
 The bubble region will be $\D_{4 \varepsilon R}$ while the neck region will be $\D_{\frac{1}{R} } \backslash \D_{\varepsilon R}$, with a $R$ that we determine in what follows.
We consider $\Lr^\varepsilon$ a first Willmore quantity of $\Phi^\varepsilon$ on $\D$.
\newline
\underline{ \bf{ Step 1: }} {\bf Neck estimates }\newline
By hypothesis \ref{hypothesis3}, there exists $R_0>0$ such that for $\varepsilon$ small enough,
$$\int_{\D_{\frac{1}{R_0}} \backslash \D_{\varepsilon R_0 } } \left| \nabla \n^\varepsilon \right| \le \varepsilon_0,$$ where $\varepsilon_0$ is given by  theorem \ref{ll2infinianneau}. In turn this ensures that
$$ \sup_{\varepsilon R_0 <s<\frac{1}{2R_0} } \int_{\D_{2s} \backslash \D_s } \left| \nabla \n^\varepsilon \right|^2 \le \varepsilon_0.$$
We can then apply theorem \ref{ll2infinianneau} and find a sequence $\vec{ \mathcal{L}_1^\varepsilon} \in \R^3$ such that 
\begin{equation}
\label{oncontrolell2infinisurlecouavantlabubulleminimale}
\left\| \left( \Lr^\varepsilon -\vec{ \mathcal{L}_1^\varepsilon} \right)e^{\lambda^\varepsilon} \right\|_{L^{2, \infty} \left( \D_{\frac{1}{2 R_0}} \backslash  \D_{2 \varepsilon R_0 } \right)} \le C \left\| H^\varepsilon \nabla \Phi^\varepsilon \right\|_{L^2 \left( \D \right)},
\end{equation}
where $C$ depends solely on $M$ defined in \ref{hypothesis1} and \ref{hypothesis2}.
\newline
\underline{ \bf{ Step 2: }} {\bf Bubble estimates }\newline
Let $p^\varepsilon = \varepsilon x^\varepsilon \in \D_{4 R_0 \varepsilon}$ and $r^\varepsilon = \varepsilon s^\varepsilon$ such that $B_{r^\varepsilon} (p^\varepsilon) \subset \D_{4 R_0 \varepsilon}$ and $$ \int_{B_{r^\varepsilon}(p^\varepsilon)} \left| \, \nabla \n^\varepsilon \right|^2= \frac{ 8 \pi}{6}.$$
 Then $x^{\varepsilon} \in \D_{4 R_0}$ and $s^\varepsilon \le 4 R_0$, meaning that there exists $x \in \overline{\D_{4 R_0}}$ and $s \le 4 R_0$ such that (up to a subsequence) 
$$\begin{aligned}
x^\epsilon &\rightarrow x,\\
s^\epsilon &\rightarrow s, \\
B_s \left( x \right)& \subset \D_{4R_0}.
\end{aligned}$$
Adapting slightly (\ref{integraledunablantilde}) we find 
$$ \lim_{\varepsilon \rightarrow 0} \int_{B_{r^\varepsilon} (p^\varepsilon)} \left| \nabla \n^\varepsilon \right|^2dz =\lim_{\varepsilon \rightarrow 0} \int_{B_{s^\varepsilon} (x^\varepsilon)} \left| \nabla \tilde \n^\varepsilon \right|^2dz  = \int _{B_s \left(x \right) } \left| \nabla \n_\Psi \right|^2dz = \frac{8\pi}{6}.$$
Necessarily $$\frac{s}{4R_0} \ge r_0^\Psi:=\frac{1}{4 R_0}  \inf \left\{ t \left| \int_{B_t(p)} \left| \, \nabla \n_\Psi \right|^2 = \frac{4 \pi}{3}, \: \forall p \in \D_{4R_0} \text{ s.t. } B_s(t) \subset \D_{4R_0} \right. \right\}>0.$$
Thus if we set $$r_0^\varepsilon:= \frac{1}{4\varepsilon R_0}  \inf \left\{ r \left| \int_{B_r(p)} \left| \, \nabla \n^\varepsilon \right|^2 = \frac{4 \pi}{3}, \: \forall p \in \D_{4\varepsilon R_0} \text{ s.t. } B_r(t) \subset \D_{4 \varepsilon R_0} \right. \right\},$$ we deduce that for $\varepsilon$ small enough $r_0^\epsilon$ is uniformly bounded from below:
\begin{equation}
\label{onborner0epsilonparenbas}
r_0^\epsilon \ge \frac{1}{10} r_0^\Psi.
\end{equation}
Inequality (\ref{onborner0epsilonparenbas}) translates the simple bubbling of $\Phi^\varepsilon$. While $\Phi^\varepsilon$ concentrates at $0$ at the scale $\varepsilon$, $\tilde \Phi^\varepsilon$ does not concentrate any further, everything happens at the same scale for $\tilde \Phi^\varepsilon$.
For instance, corollary \ref{controlconformalfactorsanspetit} ensures that the conformal factor satisfies a Harnack inequality. Namely we find $\Lambda^\varepsilon \in \R$ such that 
\begin{equation}
\label{leharnacksurlabubulle}
\forall x \in \D_{3 \varepsilon R_0} \quad \frac{e^{\Lambda^\varepsilon} }{ C} \le e^{\lambda^\varepsilon (x) } \le C e^{\Lambda^\varepsilon}.
\end{equation}  
Here $C$ depends on $M$ and $r_0^\Psi$. Theorem \ref{epsilonreularitylintro} then allows us to control the first Willmore quantity ; 
i.e. there exists $\vec{\mathcal{L}_2^\varepsilon} \in \R^3$ such that 
\begin{equation}
\label{lecontroledevecLsurlabubulleauneconstante}
\left\| \left( \Lr^\varepsilon -\vec{\mathcal{L}_2^\varepsilon} \right) e^{\lambda^\varepsilon} \right\|_{L^{2,\infty} \left( \D_{3 \varepsilon R_0} \right)} \le C(M, r_0^\Psi ) \left\| H^\varepsilon \nabla \Phi^\varepsilon \right\|_{L^2 \left( \D \right) }. 
\end{equation}
\newline
\underline{ \bf{ Step 3: }} {\bf Estimates across the concentration point}\newline
We first wish to estimate $\left| \vec{\mathcal{L}_1^\varepsilon} -  \vec{\mathcal{L}_2^\varepsilon} \right|$. Using  (\ref{oncontrolell2infinisurlecouavantlabubulleminimale}) and (\ref{lecontroledevecLsurlabubulleauneconstante}) we find
$$\begin{aligned}
\left\| \left( \vec{\mathcal{L}_1^\varepsilon} -  \vec{\mathcal{L}_2^\varepsilon} \right) e^{\lambda^\varepsilon} \right\|_{L^{2, \infty} \left( \D_{3 R_0 \varepsilon} \backslash \D_{2R_0 \varepsilon} \right)} &\le \left\| \left( \vec{\mathcal{L}_1^\varepsilon} -  \Lr^\varepsilon \right) e^{\lambda^\varepsilon}\right\|_{L^{2, \infty} \left( \D_{3 R_0 \varepsilon} \backslash \D_{2R_0 \varepsilon} \right)}  \\&
 + \left\| \left( \Lr^\varepsilon -  \vec{\mathcal{L}_2^\varepsilon} \right) e^{\lambda^\varepsilon} \right\|_{L^{2, \infty} \left( \D_{3 R_0 \varepsilon} \backslash \D_{2R_0 \varepsilon} \right)} \\
&\le \left\| \left( \vec{\mathcal{L}_1^\varepsilon} -  \Lr^\varepsilon \right) e^{\lambda^\varepsilon} \right\|_{L^{2, \infty} \left( \D_{\frac{1}{2R_0}} \backslash \D_{2R_0 \varepsilon} \right)} \\& + \left\| \left( \Lr^\varepsilon -  \vec{\mathcal{L}_2^\varepsilon} \right) e^{\lambda^\varepsilon} \right\|_{L^{2, \infty} \left( \D_{3 R_0 \varepsilon}  \right)}  \\
&\le C( M, r^\Psi_0 ) \left\| H^\epsilon \nabla \Phi^\epsilon \right\|_{L^2 \left( \D \right) }.
\end{aligned}$$

Thus 
\begin{equation}
\label{oncontroleladifferencedeslcestbietotofini}
\left| \vec{\mathcal{L}_1^\varepsilon} -  \vec{\mathcal{L}_2^\varepsilon} \right| \le \frac{ C( M, r^\Psi_0 )}{ \left\| e^{\lambda^\varepsilon} \right\|_{L^{2, \infty} \left( \D_{3 R_0 \varepsilon} \backslash \D_{2R_0 \varepsilon} \right)}} \left\| H^\epsilon \nabla \Phi^\varepsilon \right\|_{L^2 \left( \D \right) }.
\end{equation}
We can now assemble our estimates on the neck and the bubble. Using successively (\ref{oncontrolell2infinisurlecouavantlabubulleminimale}), (\ref{lecontroledevecLsurlabubulleauneconstante}) and (\ref{oncontroleladifferencedeslcestbietotofini}) we find
$$\begin{aligned}
\left\| \left( \Lr^\varepsilon - \vec{\mathcal{L}_1^\varepsilon } \right)e^{\lambda^\varepsilon} \right\|_{L^{2,\infty} \left( \D_{\frac{1}{2 R_0}} \right)} &\le  \left\| \left( \Lr^\varepsilon -\vec{ \mathcal{L}_1^\varepsilon} \right)e^{\lambda^\varepsilon} \right\|_{L^{2, \infty} \left( \D_{\frac{1}{2 R_0}} \backslash  \D_{2 \varepsilon R_0 } \right)}  \\&+ \left\| \left( \Lr^\varepsilon -\vec{ \mathcal{L}_1^\varepsilon} \right)e^{\lambda^\varepsilon} \right\|_{L^{2, \infty} \left( \D_{3 \varepsilon R_0 } \right)}  \\
&\le C( M) \left\| H^\varepsilon \nabla \Phi^\varepsilon \right\|_{L^2 \left( \D \right)} \\&+ \left\| \left( \Lr^\varepsilon -\vec{ \mathcal{L}_2^\varepsilon} \right)e^{\lambda^\varepsilon} \right\|_{L^{2, \infty} \left( \D_{3 \varepsilon R_0 } \right)} \\&+ \left\| \left( \vec{\mathcal{L}_2^\varepsilon} -\vec{ \mathcal{L}_1^\varepsilon} \right)e^{\lambda^\varepsilon} \right\|_{L^{2, \infty} \left( \D_{3 \varepsilon R_0 } \right)} \\
&\le  C( M, r^\Psi_0 ) \left\| H^\varepsilon \nabla \Phi^\varepsilon \right\|_{L^2 \left( \D \right)} + \left| \vec{\mathcal{L}_1^\varepsilon} -  \vec{\mathcal{L}_2^\varepsilon} \right|  \left\| e^{\lambda^\varepsilon} \right\|_{L^{2, \infty} \left( \D_{3 \varepsilon R_0 } \right)} \\
&\le C( M, r^\Psi_0 ) \left\| H^\varepsilon \nabla \Phi^\varepsilon \right\|_{L^2 \left( \D \right)}  \left( 1 + \frac{  \left\| e^{\lambda^\varepsilon} \right\|_{L^{2, \infty} \left( \D_{3 \varepsilon R_0 } \right)}  }{ \left\| e^{\lambda^\varepsilon} \right\|_{L^{2, \infty} \left( \D_{3 R_0 \varepsilon} \backslash \D_{2R_0 \varepsilon} \right)}}  \right).
\end{aligned}$$
With  (\ref{leharnacksurlabubulle}),  we can simplify the last right-hand term in the inequality.
$$ \begin{aligned}
\frac{  \left\| e^{\lambda^\varepsilon} \right\|_{L^{2, \infty} \left( \D_{3 \varepsilon R_0 } \right)}  }{ \left\| e^{\lambda^\varepsilon} \right\|_{L^{2, \infty} \left( \D_{3 R_0 \varepsilon} \backslash \D_{2R_0 \varepsilon} \right)}} &\le  C( M, r^\Psi_0 )  \frac{  \left\| e^{\Lambda^\varepsilon} \right\|_{L^{2, \infty} \left( \D_{3 \varepsilon R_0 } \right)}  }{ \left\| e^{\Lambda^\varepsilon} \right\|_{L^{2, \infty} \left( \D_{3 R_0 \varepsilon} \backslash \D_{2R_0 \varepsilon} \right)}} \\
&\le C( M, r^\Psi_0 )  \frac{  \left\| 1\right\|_{L^{2, \infty} \left( \D_{3 \varepsilon R_0 } \right)}  }{ \left\|1 \right\|_{L^{2, \infty} \left( \D_{3 R_0 \varepsilon} \backslash \D_{2R_0 \varepsilon} \right)}}  \\
&\le  C( M, r^\Psi_0 ) \end{aligned}$$
since $\Lambda^\varepsilon$ is a constant.
Accordingly,  there exists  $C( M, r^\Psi_0 ) >0$ such that the following estimate across the concentration point stands.
\begin{equation}
\label{ontrouvelebonL}
\left\| \left( \Lr^\varepsilon - \vec{\mathcal{L}_1^\varepsilon } \right)e^{\lambda^\varepsilon} \right\|_{L^{2,\infty} \left( \D_{\frac{1}{2 R_0}} \right)} \le   C( M, r^\Psi_0 ) \left\| H^\varepsilon \nabla \Phi^\varepsilon \right\|_{L^2 \left( \D \right)}.
\end{equation}
\newline
\underline{ \bf{ Step 4: }} {\bf Conclusion}\newline
We have then found a first Willmore quantity, $\Lr^\varepsilon - \vec{ \mathcal{L}_1^\varepsilon}$, with uniform $L^{2,\infty}$ control on a disk of fixed radius $\rho =\frac{1}{2R_0} $. Since (\ref{l'hypotheseduhpetit}) stands  we can apply theorem \ref{epsilonregularitehnablaphifaibleLL2inftyi} on $\D_\rho$ and find 
\begin{equation} \label{lestimeefinaleenH} \left\| H^\varepsilon \nabla \Phi^\varepsilon \right\|_{L^\infty \left( \D_{ \frac{ \rho}{2} } \right)} \le   C { \left\| H^\varepsilon \nabla \Phi^\varepsilon \right\|_{L^2 \left( \D_\rho \right)}},\end{equation}
\begin{equation} \label{lestimeefinaleenPhi} \left\| \nabla \Phi^\varepsilon \right\|_{W^{2,p} \left( \D_{ \frac{\rho}{2} } \right)} \le      C  \| \nabla \Phi^\varepsilon \|_{L^2 \left( \D_\rho \right) }, \end{equation}
while the second and third Willmore quantities satisfy
\begin{equation} \label{lestimeefinaleenSR}\| \nabla S^\varepsilon \|_{W^{1,p} \left( \D_{\frac{\rho}{2}} \right) } + \| \nabla \vec{R}^\varepsilon \|_{W^{1,p} \left( \D_{\frac{\rho}{2}} \right) }\le C   \| H^\varepsilon \nabla \Phi^\varepsilon  \|_{L^2 \left( \D_\rho \right)}\end{equation} for all $p<\infty$. 

Theorem \ref{convminibubbleintro} then follows from classical compactness results.
\qed

\renewcommand{\thesection}{\Alph{section}}
\setcounter{section}{0} 
\section{Appendix}
\subsection{Formulas for a conformal immersion}
\label{formulasconformalimmersion}
In this section, we show several formulas useful for the core of the article. Most are well known, but their proof is included for self-containedness.
Let $\Phi \,: \, \D \rightarrow \R^3$ be a conformal immersion, that is such that 
$$\left| \Phi_x \right|^2 - \left| \Phi_y \right|^2 = \left\langle \Phi_x ,  \Phi_y \right\rangle = 0.$$ 
Its Gauss map is defined  as $\n = \frac{ \Phi_x \times  \Phi_y}{\left|  \Phi_x \times  \Phi_y \right|}$ (with $\times$ the usual vectorial product in $\R^3$) and its conformal factor  as $\lambda = \log \left| \Phi_x \right| = \log \left| \Phi_y \right|$.
Its second fundamental form is then $$ A:= \left\langle \nabla^2 \Phi , \n \right\rangle =: \begin{pmatrix} e & f \\ f & g \end{pmatrix}.$$
One can check \begin{equation} \label{lenablanenannex} \nabla \n = - e^{-2 \lambda} A \nabla \Phi =- e^{-2 \lambda} \begin{pmatrix} e \Phi_x + f \Phi_y \\ f \Phi_x + g \Phi_y \end{pmatrix} \end{equation} and deduce immediately 
\begin{equation}
\label{lenablaperpnenannex}
\nabla^{\perp} \n = - e^{-2 \lambda} \begin{pmatrix} -f \Phi_x  -g \Phi_y \\ e \Phi_x + f \Phi_y \end{pmatrix} .
\end{equation}
Defining the mean curvature $$H = \frac{e+g}{2 e^{2 \lambda} }$$ and the tracefree second fundamental form $$\Ar =e^{-2\lambda} \begin{pmatrix} \frac{e-g}{2} & f \\ f & \frac{g-e}{2} \end{pmatrix}$$ one finds
\begin{equation}
\label{lenablalenablaperpenannexenHAr}
\begin{aligned}
\nabla \n &= -H \nabla \Phi - \Ar \nabla \Phi, \\
\nabla^\perp \n &= - H \nabla^\perp \Phi + \Ar \nabla^\perp \Phi.
\end{aligned}
\end{equation}
By definition  of $\n$ 
$$\begin{aligned}
\n \times \Phi_x &= \Phi_y, \\
\n \times \Phi_y &= - \Phi_x,
\end{aligned}$$
which implies
\begin{equation}
\label{ntimesnablanablaperpphi}
\begin{aligned}
\n \times \nabla \Phi &=  - \nabla^\perp \Phi, \\
\n \times \nabla^\perp \Phi &= \nabla \Phi.
\end{aligned}
\end{equation}
Combining (\ref{lenablalenablaperpenannexenHAr}) and (\ref{ntimesnablanablaperpphi}) yields 
\begin{equation}
\label{ntimesnablanableperpn}
\begin{aligned}
\n  \times \nabla \n &= H \nabla^\perp \Phi + \Ar \nabla^\perp \Phi, \\
\n \times \nabla^\perp \n &= - H \nabla \Phi + \Ar \nabla \Phi.
\end{aligned}
\end{equation}
As a result $H \nabla \Phi$ and $\Ar \nabla \Phi$ can be deduced solely from $\nabla \n$: 
\begin{equation}
\label{HnablaphietArnablaphienfonctionden}
\begin{aligned}
H \nabla \Phi &=- \frac{ \n \times \nabla^\perp \n  + \nabla \n }{2}, \\
\Ar \nabla \Phi &= \frac{ \n \times \nabla^\perp \n - \nabla \n }{2}.
\end{aligned}
\end{equation}
It is well known that, since $\Phi$ is conformal, 
\begin{equation}
\label{leDeltaPhinormal}
\Delta \Phi = \vec{H} \left| \nabla \Phi \right|^2,
\end{equation}
 where $\vec{H} = H \n$, 
and \begin{equation}\label{Liouvilleequation}\Delta \lambda = Ke^{2\lambda},\end{equation} where $K =e^{-4 \lambda} \det A= e^{-4 \lambda} \left( eg-f^2 \right) $ is the Gauss curvature. Equation (\ref{Liouvilleequation}) is refered to as the Liouville equation.

We can compute $\left| \nabla \n \right|^2$ in several ways.
Using (\ref{lenablanenannex}): 
\begin{equation}
\label{normenablan2enHetK} \begin{aligned}
\left| \nabla \n \right|^2 &= e^{-2 \lambda} \left( e^2 + g^2 + 2 f^2 \right) \\
&=  e^{-2 \lambda} \left( 4 \left(\frac{e+g}{2} \right)^2  - 2 eg+ 2 f^2 \right) \\
&= 2 \left( 2H^2  -  K \right) e^{2 \lambda},
\end{aligned}
\end{equation}
and with (\ref{lenablalenablaperpenannexenHAr})
\begin{equation}
\label{normenalan2enHetAr}
\begin{aligned}
\left| \nabla \n \right|^2 &= \left| H \nabla \Phi + \Ar \nabla \Phi \right|^2 \\
&=  \left|H \nabla \Phi \right|^2 + \left| \Ar \nabla \Phi \right|^2 \text{ since } \Ar \text{ is tracefree.} \\
\end{aligned}
\end{equation}
 since $\Ar$ is tracefree.
From (\ref{normenalan2enHetAr}) we deduce 
\begin{equation}
\label{Arborneparnablanannex}
\big|\Ar\big|\le \left|\nabla \n \right|, 
\end{equation}
and
\begin{equation}
\label{HnablaPhiborneparnablanannex}
\left| H \nabla \Phi \right| \le \left| \nabla \n \right|.
\end{equation}

Complex notations will prove convenient for this work. Let $$\partial_z = \frac{1}{2} \left( \partial_x - i \partial_y \right) = \frac{1}{2}\begin{pmatrix} 1 \\ -i \end{pmatrix} . \nabla =  \frac{i}{2}\begin{pmatrix} 1 \\ -i \end{pmatrix} . \nabla^\perp.$$
Then $\Phi$ conformal translates as 
\begin{equation}
\label{conformalcondition}
\begin{aligned}
&\left\langle \Phi_z, \Phi_z \right\rangle = 0, \\
& \left| \Phi_z \right|^2 = \frac{e^{2\lambda} }{2}.
\end{aligned}
\end{equation}
If we define the tracefree curvature as $\omega = \frac{e-g}{2} -if =2 \left\langle \Phi_{zz}, \n \right\rangle$,  (\ref{lenablalenablaperpenannexenHAr}) becomes 
\begin{equation}
\label{lenablalenablaperpenannexenHomegacomplexe}
\begin{aligned}
\n_z &= -H \Phi_z - \omega e^{-2\lambda}  \Phi_{\zb},
\end{aligned}
\end{equation}
while (\ref{normenalan2enHetAr}) turns into
\begin{equation}
\label{nzencomplexesenhetomega}
\left| \n_z \right|^2 = \frac{H^2 e^{2\lambda} + \left|\omega \right|^2 e^{-2\lambda} }{2}.
\end{equation}
Similarly, (\ref{leDeltaPhinormal}) translates to
$$ \Phi_{z\zb} = \frac{He^{2\lambda} }{2}\n.$$

Exploiting (\ref{conformalcondition}) one finds 
$$
\begin{aligned}
\left\langle \Phi_{zz}, \Phi_z \right\rangle &=0 \\
\left\langle \Phi_{zz}, \Phi_{\zb} \right\rangle&=\left(  \left\langle \Phi_z , \Phi_{\zb} \right\rangle \right)_z - \left\langle \Phi_z, \Phi_{z \zb} \right\rangle \\
&= \lambda_z e^{2 \lambda}.
\end{aligned}$$
Subsequently 
$$ \Phi_{zz} = 2 \lambda_z \Phi_z +\frac{ \omega }{2}\n.$$
We can then compute
$$ \begin{aligned}
\n_{z \zb} &= - H_{\zb} \Phi_z -\frac{H^2 e^{2\lambda} }{2} \n - \left( \omega_{\zb} e^{-2\lambda} - 2 \lambda_{\zb} \omega e^{-2 \lambda} \right) \Phi_{\zb} - 2 \lambda_{\zb} \left( \omega e^{-2 \lambda} \right) \Phi_{\zb} - \frac{ \left| \omega \right|^2 e^{-2 \lambda} }{2} \n \\
&=- H_{\zb} \Phi_z - \omega_{\zb} e^{-2\lambda} \Phi_{\zb}  - \frac{ H^2 e^{2\lambda} + \left| \omega \right|^2 e^{-2\lambda} }{2} \n .
\end{aligned}$$
However $\n_{z \zb} \in \R^3$
since $\n \in \R^3$. Then necessarily $\omega_{\zb} e^{-2\lambda} = \overline{ H_{\zb} }$ i.e.
\begin{equation}
\label{Gausscodazziencomplexes}
H_z = \omega_{\zb} e^{-2\lambda}.
\end{equation}
Equation (\ref{Gausscodazziencomplexes}) is the Gauss-Codazzi equation in complex notations.

Using (\ref{Gausscodazziencomplexes}) and (\ref{nzencomplexesenhetomega})  we find 
\begin{equation}
\label{equationnzzbannexarticle}
\n_{z \zb} + \left| \n_z \right|^2 \n + 2 \Re \left( H_z \Phi_{\zb} \right)=0.
\end{equation}
While the complex notations are most convenient for computations, the resulting equations are not always telling. We will then translate (\ref{equationnzzbannexarticle}) back to its classic real form: 
$$ \begin{aligned}
\n_{z \zb} + \left| \n_z \right|^2 \n + 2 \Re \left( H_z \Phi_{\zb} \right) = \frac{1}{4} \left( \Delta \n + \left| \nabla \n \right|^2  \n +2 \left( H_x \Phi_x + H_y \Phi_y \right) \right).
\end{aligned}$$
The Gauss map $\n$ then satisfies 
\begin{equation}
\label{equationDeltanannexarticle1}
\Delta \n + \left| \nabla \n \right|^2  \n +2 \nabla H \nabla \Phi=0.
\end{equation}

This can be slightly changed to better suit our needs 
$$ \begin{aligned}
\Delta \n + \left| \nabla \n \right|^2  \n +2 \nabla H \nabla \Phi &=\Delta \n + \left| \nabla \n \right|^2  \n + 2 \mathrm{div}\left( H \nabla \Phi \right) - 2H \Delta \Phi   \\
&= \Delta \n + \left(  \left| \nabla \n \right|^2 - 2 \left| H \nabla \Phi \right|^2 \right)  \n + 2 \mathrm{div}\left( H \nabla \Phi \right)  \\
&= \Delta \n + \left(  \left| \Ar \nabla \Phi \right|^2 -  \left| H \nabla \Phi \right|^2 \right)  \n + 2 \mathrm{div}\left( H \nabla \Phi \right). 
\end{aligned}$$
The second equality is obtained with (\ref{leDeltaPhinormal}), and the third  with (\ref{normenalan2enHetAr}).
Now we compute 
$$\begin{aligned}
\nabla \n \times \nabla^\perp \n &=- \n_x \times \n_y + \n_y \times \n_x =-2 \n_x \times \n_y \\
&= - 2 e^{-4\lambda} \left( e \Phi_x + f \Phi_y \right) \times \left( f \Phi_x + g \Phi_y \right) \\
&= -2 e^{-2\lambda} \left( eg \n -f^2 \n \right) \\
&= - 2 e^{-2 \lambda} \left( \left( \frac{e+g}{2} \right)^2 - \left( \frac{e-g}{2} \right)^2 - f^2 \right) \n \\
&= -2 H^2 e^{2\lambda} \n  +  2 \left( \left( \frac{e-g}{2} \right)^2 + f^2 \right)e^{-2\lambda} \n \\
&= - \left| H \nabla \Phi \right|^2 \n + \left| \Ar \nabla \Phi \right|^2\n.
\end{aligned}$$
We then find 
\begin{equation}
\label{equationDeltanannexarticlemieux!}
\Delta \n + \nabla^\perp \n \times \nabla \n  +2 \mathrm{div}\left( H \nabla \Phi \right) = 0.
\end{equation}

\subsection{Formulas for a conformal, Willmore immersion }
\label{formulasforaconformalWillmoreimmersion}
The aim of this section is to study the Willmore quantities and offer a proof of theorem \ref{theoRSPhi}. To that aim, we set ourselves in the same context as in the previous subsection with the additionnal assumption that $\Phi$ is Willmore.

We recall the definition of the Willmore quantities (already introduced in section \ref{subsectionWillmorequantities} and stemming from theorem I.4 in \cite{bibanalysisaspects}).
\begin{equation}
\label{willmorequantitiesdefinitionannex}
\begin{aligned}
 \nabla^\perp \Lr &= \nabla \vec{H} - 3 \pi_{\n} \left( \nabla \vec{H} \right) + \nabla^\perp \n \times \vec{H}, \\
\nabla^\perp S &= \left\langle \Lr, \nabla^\perp \Phi \right\rangle, \\
\nabla^\perp \vec{R} &= \Lr \times \nabla^\perp \Phi + 2H \nabla^\perp \Phi.
\end{aligned}
\end{equation}
Now since $\left\langle \Lr \times \nabla^\perp \Phi , H\nabla^\perp \Phi \right\rangle = 0$ owing to the properties of the vectorial product, we can compute 
$$\left| \nabla \vec{R} \right|^2 = \left| \Lr \times \nabla^\perp \Phi  \right|^2 + 4 \left| H \nabla \Phi \right|^2.$$ 
This yields an interesting estimate: 
\begin{equation}
\label{majorerhnablaphiparnablaR}
\left| H \nabla \Phi \right| \le \frac{1}{2} \left| \nabla \vec{R} \right|. 
\end{equation}
We recall the system of \ref{theoremequationRS}
$$
\left\{
\begin{aligned}
\Delta S &= - \left\langle \nabla \n, \nabla^\perp \vec{R} \right\rangle \\
\Delta \vec{R} &= \nabla \n \times \nabla^\perp \vec{R} + \nabla^\perp S \nabla \n\\
\Delta \Phi &= \frac{1}{2} \left( \nabla^\perp S. \nabla \Phi + \nabla^\perp \vec{R} \times \nabla \Phi \right).
\end{aligned}
\right.
$$

To rephrase this system we compute
\begin{equation} \label{ArnablaPhiscalarnablaperpR} \begin{aligned}
\left\langle \Ar \nabla \Phi , \nabla^\perp \vec{R} \right\rangle &= \left\langle \Ar \nabla \Phi , \Lr \times \nabla^\perp \Phi + 2 H \nabla^\perp \Phi \right\rangle \\
&=-e^{-2\lambda} \left\langle \frac{e-g}{2} \Phi_x + f \Phi_y , \Lr \times \Phi_y +2 H \Phi_y \right\rangle \\&+ e^{-2 \lambda} \left\langle f \Phi_x + \frac{g-e}{2} \Phi_y , \Lr \times \Phi_x + 2H \Phi_x \right\rangle \\
&=\frac{g-e}{2} e^{-2\lambda} \left( \left\langle \Phi_x, \Lr \times \Phi_y \right\rangle + \left\langle \Phi_y, \Lr \times \Phi_x \right\rangle \right)
-2H f  + 2 Hf \\
&= \frac{g-e}{2} e^{-2\lambda} \left( \left\langle \Lr , \Phi_y \times \Phi_x \right\rangle + \left\langle \Lr, \Phi_x \times \Phi_y \right\rangle \right) \\
&=0.
\end{aligned} \end{equation}
Further 
\begin{equation} \label{ArnablaPhitimesnablaperpR} \begin{aligned}
\Ar\nabla\Phi \times \nabla^\perp \vec{R} &= \Ar \nabla \Phi \times \left( \Lr \times \nabla^\perp \Phi + 2 H \nabla^\perp \Phi \right) \\
&= \left\langle \Ar \nabla \Phi . \nabla^\perp \Phi \right\rangle \Lr - \left\langle \Ar \nabla \Phi, \Lr \right\rangle \nabla^\perp \Phi +2H \Ar \nabla\Phi \times \nabla^\perp \Phi \\
&=  -e^{-2\lambda} \left\langle \frac{e-g}{2} \Phi_x+f\Phi_y, \Phi_y \right\rangle \Lr \\ &+ e^{-2\lambda} \left\langle \frac{e-g}{2} \Phi_x+f\Phi_y, \Lr \right\rangle \Phi_y \\& -e^{-2\lambda} 2H \left( \frac{e-g}{2} \Phi_x +f\Phi_y \right) \times \Phi_y  \\&+ 
e^{-2\lambda} \ \left\langle f\Phi_x + \frac{g-e}{2} \Phi_y, \Phi_x \right\rangle \Lr \\&  - e^{-2\lambda} \left\langle f\Phi_x + \frac{g-e}{2} \Phi_y, \Lr \right\rangle \Phi_x  \\&+2H \left( f\Phi_x + \frac{g-e}{2} \Phi_y \right) \times \Phi_x  \\
&=  \frac{e-g}{2} \left\langle \Phi_x , \Lr \right\rangle \Phi_y + f \left\langle \Phi_y , \Lr \right\rangle \Phi_y \\&- f \left\langle \Phi_x , \Lr \right\rangle \Phi_x -\frac{g-e}{2} \left\langle \Phi_y, \Lr \right\rangle \Phi_x \\ & +2H \left( -\frac{e-g}{2} \n  + \frac{g-e}{2} (-\n) \right)  \\
&=\left( \frac{e-g}{2} S_x \Phi_y + f S_y \Phi_y - f S_x \Phi_x - \frac{g-e}{2}S_y \Phi_x \right), \\
&=- \nabla^\perp S \Ar \nabla \Phi.
\end{aligned} \end{equation}
 We have used (\ref{willmorequantitiesdefinitionannex}) to obtain the second to last equality.
The decomposition (\ref{lenablalenablaperpenannexenHAr}) then yields
$$\left\langle \nabla \n, \nabla^\perp \vec{R} \right\rangle = -\left\langle  H \nabla \Phi + \Ar \nabla\Phi, \nabla^\perp \vec{R} \right\rangle = - \left\langle H \nabla \Phi , \nabla^\perp \vec{R} \right\rangle,$$
with (\ref{ArnablaPhiscalarnablaperpR}). Similarly with (\ref{ArnablaPhitimesnablaperpR}) we compute
$$
\begin{aligned}
\nabla \n \times \nabla^\perp \vec{R} + \nabla^\perp  S \nabla \n &= - H\nabla \Phi \times \nabla^\perp \vec{R} - H \nabla \Phi \nabla^\perp S - \Ar \nabla \Phi \nabla^\perp \vec{R} - \Ar \nabla \Phi \nabla^\perp S \\
&=  - H\nabla \Phi \times \nabla^\perp \vec{R} - H \nabla \Phi \nabla^\perp S + \nabla^\perp S \Ar\nabla \Phi - \nabla^\perp S \Ar \nabla \Phi \\
&= - H\nabla \Phi \times \nabla^\perp \vec{R} - H \nabla \Phi \nabla^\perp S .
\end{aligned}$$
Injecting these last two equalities in (\ref{systemenRSPhiannex}), we can conclude that $\vec{R}$, $S$ and $\Phi$ satisfy: 
$$
\left\{
\begin{aligned}
\Delta S &=  \left\langle H \nabla \Phi, \nabla^\perp \vec{R} \right\rangle \\
\Delta \vec{R} &= - H \nabla \Phi \times \nabla^\perp \vec{R} - \nabla^\perp S H \nabla \Phi\\
\Delta \Phi &= \frac{1}{2} \left( \nabla^\perp S. \nabla \Phi + \nabla^\perp \vec{R} \times \nabla \Phi \right),
\end{aligned}
\right.
$$
which is the desired equation.\qed

\subsection{Low-regularity estimates}
We first recall theorem 3 of \cite{MR1949165}.

\begin{theo}
\label{231120201048}
Let $f \in L^p \left( \D \right)$ such that $\int f dxdy= 0$, $1<p< \infty$. Then there exists some $Y \in L^\infty \left( \D \right) \cap W_0^{1,p}\left( \D \right)$ such that 
$$\mathrm{div} Y = f,$$
and:
$$\left\| Y \right\|_{L^\infty \left( \D \right) } + \left\| \nabla Y \right\|_{W^{1,p}(\D )} \le C \left\| f \right\|_{L^p \left( \D \right) }.$$
\end{theo}

\begin{theo}
\label{theoremepourbornerL}
Let $V \in \mathcal{D}' \left( \R^3 \right)$ such that $ \nabla V =   \nabla^\perp A + B$ with $A  \in L^{2} \left(\D\right) $ and $B  \in L^1 \left( \D, \R^2 \right)$.
Then for any $r<1$ there exists $c \in \R$ a constant and $C(r)>0$  such that 
$$ \left\| V-c \right\|_{L^{2} \left( \D_{r} \right) } \le C(r) \left( \| A \|_{L^2 \left( \D \right) }+ \| B \|_{L^1 \left( \D \right)} \right).$$
\end{theo}

\begin{proof}
Let us first assume that  $V$, $A$ and $B$ are smooth, and obtain an a priori estimate.
 Then, for any $U  = \begin{pmatrix} U_1 \\ U_2 \end{pmatrix} \in W^{1,2}_0 (\D, \R^2)$, one has:
\begin{equation} \label{231120201051}\begin{aligned}
\left|\int_\D \nabla V. U dx dy \right| & \le \left| \int_\D \nabla^\perp A . U dxdy \right| + \left| \int_\D B . U dxdy \right| \\
&\le  \left| \int_\D A . \mathrm{div} \left( \begin{pmatrix} U_2 \\ -U_1 \end{pmatrix} \right) dxdy \right|+  \left| \int_\D B . U dxdy \right|  \\
&\le \left( \left\| A \right\|_{L^2 \left( \D \right)} + \left\| B \right\|_{L^1 \left( \D \right)} \right) \left( \left\| U \right\|_{L^\infty \left( \D \right)} + \left\| \nabla U \right\|_{L^2 \left( \D \right) } \right).
\end{aligned} \end{equation}
Applying theorem \ref{231120201048} with $f= V - \bar{V}$, with $\bar{V}= \frac{1}{\left| \D \right|} \int_\D Vdxdy$, we can find $U \in L^\infty \left( \D \right) \cap W^{1,2}_0 \left( \D , \R^2 \right)$ such that:
\begin{equation}
\label{eqdivU}
\mathrm{div} \left(\mathrm{U} \right)= V -\bar{V},
\end{equation}
and 
\begin{equation} \label{231120201055} \left\| U \right\|_{L^\infty \left( \D \right) } + \left\| \nabla U \right\|_{W^{1,p}(\D )} \le C \left\| V -\bar{V} \right\|_{L^2 \left( \D \right) }.\end{equation}
Then, integrating by parts yields: 
\begin{equation}
\label{231120201104}
\begin{aligned}
\left| \int_\D  \left( V -\bar{V} \right)^2 dxdy \right| &\le \left| \int_\D \left( V -\bar{V} \right) \mathrm{div} U dxdy \right| \\
&\le  \left| \int_\D \nabla \left( V -\bar{V} \right) U dxdy \right| \le \left| \int_\D \nabla  V U dxdy \right|
\\
&\le  \left( \left\| A \right\|_{L^2 \left( \D \right)} + \left\| B \right\|_{L^1 \left( \D \right)} \right) \left( \left\| U \right\|_{L^\infty \left( \D \right)} + \left\| \nabla U \right\|_{L^2 \left( \D \right) } \right) \\
&\le C \left( \left\| A \right\|_{L^2 \left( \D \right)} + \left\| B \right\|_{L^1 \left( \D \right)} \right) \left\| V -\bar{V} \right\|_{L^2 \left( \D \right) },
\end{aligned}
\end{equation}
injecting first \eqref{eqdivU}, then \eqref{231120201051}  and \eqref{231120201055}. From \eqref{231120201104}, one then deduces:

\begin{equation} 
\label{231120201105}
\left\| V- \bar{V} \right\|_{L^2 \left( \D \right)} \le C \left( \left\| A \right\|_{L^2 \left( \D \right)} + \left\| B \right\|_{L^1 \left( \D \right)} \right) .
\end{equation}
A rescaling yields \eqref{231120201105} on any $\D_r$.

In the general case, we approximate $A$, $B$ and $V$ by smooth functions on a smaller disk $\D_r$. On this smaller disk, we apply and converge in the rescaled a priori estimate \eqref{231120201105} to obtain the result.
\end{proof}

We conclude this subsection by recalling an extension of Calderon-Zygmund  with Lorentz spaces (theorem 3.3.6 in \cite{bibharmmaps}).
\begin{theo}
\label{calderonzygmundl1l2infini}
Let $\Omega$ be an open subset of $\R^2$ with $C^1$ boundary. Let $f \in L^1 \left( \Omega \right)$ and $\varphi$ solution of
$$\left\{ \begin{aligned} \Delta \varphi &= f \text{ in } \Omega \\
\varphi &= 0 \text{ on } \partial \Omega.\end{aligned} \right.$$
then there exists a constant $C\left( \Omega \right)$ such that 
$$\left\| \varphi \right\|_{L^{2,\infty} \left( \Omega \right)} \le C \left( \Omega \right) \left\| f \right\|_{L^1 \left( \Omega \right) }.$$
\end{theo}

\subsection{Wentes' lemmas}
\label{subsectionwenteslemmas}
Following are a few variations on Wente's inequality, which will prove useful in the core of the article.

\begin{theo}[Wente's inequality, originally in \cite{MR0268743}, see also 3.1.2 in \cite{bibharmmaps}]
\label{wenteclassique}
Let $a$,$b \in W^{1,2} \left( \D, \R \right)$ and $u$ a solution of 
$$\left\{ \begin{aligned} &\Delta u = \nabla a. \nabla^\perp b \text{ in } \D \\
&u = 0 \text{ on } \partial \D.
\end{aligned}
\right.$$
Then $u \in C^0 \left( \D, \R \right) \cap W^{1,2} \left( \D, \R \right)$ and there exists $C>0$
$$ \left\| u \right\|_{L^\infty \left( \D \right) } + \left\| \nabla u \right\|_{L^2 \left( \D \right) } \le C  \left\| \nabla a \right\|_{L^2 \left( \D \right)} \left\| \nabla b \right\|_{L^2 \left( \D \right)}.$$ 
\end{theo}

\begin{theo}[Wente's inequality $L^{2, \infty}$, theorem 3.4.5 of \cite{bibharmmaps}]
\label{wentel2infini}
Let $\Omega$ be a bounded domain of $\R^2$, with $C^2$ boundary. Suppose $a$ and $b$ such that $\nabla a \in L^{2, \infty} \left( \Omega \right)$ and $\nabla b \in L^2 \left( \Omega \right)$. Let $\varphi$ be the solution of 
$$\left\{ \begin{aligned} \Delta \varphi &= \nabla a. \nabla^\perp b \text{ in } \Omega \\
\varphi &= 0 \text{ on } \partial \Omega.\end{aligned} \right.$$
Then $\varphi \in W^{1,2} \left( \Omega \right)$, and there exists $C(\Omega) >0$ such that 
$$ \left\| \nabla \varphi \right\|_{L^2 \left( \Omega \right) } \le C( \Omega ) \| \nabla a \|_{L^{2, \infty} \left( \Omega \right)} \| \nabla b \|_{L^2 \left( \Omega \right) }.$$
\end{theo}

\begin{theo}[Wente's inequality $L^{2, 1}$, theorem 3.4.1 of \cite{bibharmmaps}]
\label{wentel21}
Let $\Omega$ be a bounded domain of $\R^2$, with $C^2$ boundary. Suppose $a$ and $b$ such that $ a \in W^{1,2} \left( \Omega \right)$ and  $b \in W^{1,2} \left( \Omega \right)$. Let $\varphi$ be the solution of 
$$\left\{ \begin{aligned} \Delta \varphi &= \nabla a. \nabla^\perp b \text{ in } \Omega \\
\varphi &= 0 \text{ on } \partial \Omega.\end{aligned} \right.$$
Then $\varphi \in W^{1,\left(2,1\right)} \left( \Omega \right)$, and there exists $C(\Omega) >0$ such that 
$$ \left\| \nabla \varphi \right\|_{L^{2,1} \left( \Omega \right) } \le C( \Omega ) \| \nabla a \|_{L^{2} \left( \Omega \right)} \| \nabla b \|_{L^2 \left( \Omega \right) }.$$
\end{theo}

\begin{remark}
One must notice that the constant $C(\Omega)$ in theorems \ref{wentel2infini} and \ref{wentel21} depends on the shape of $\Omega$, but not its size due to the fact that $L^{2,\infty}$ and $L^{2,1}$ are scale-invariant, but not conformal invariant.  The same constant $C$ then works for all disks $\D_r$. Since $L^2$ is  a  conformal invariant the constant in theorem \ref{wenteclassique} does not depend on $\Omega$. We refer the reader to \cite{MR1239099} for more details.
\end{remark}

\subsection{Hodge decomposition }
In this subsection we briefly recall results on the Hodge decomposition and recast them in our framework.

\begin{theo}[$L^p$ decomposition, theorem 10.5.1 in \cite{geometricfunctiontheory}]
\label{hodgedecompositionforme}
Let $\Omega$ be a smoothly bounded domain in $\R^n$ and  $1<p<\infty$. Then for any $l$-differential form $\omega \in L^p$ there exists  a $l-1$ differential form $\alpha$, a $l+1$-differential form $\beta$ and a $l$-differential form $h$ such that: 
$$ \omega = d \alpha  + d^* \beta +h$$ 
with $dh = d^* h = 0$ and 
$$ \left\| \alpha \right\|_{W^{1,p} \left( \Omega \right) } +  \left\| \beta \right\|_{W^{1,p} \left( \Omega \right) } \le C_p\left( \Omega \right) \left\| \omega \right\|_{L^{p} \left( \Omega \right)}.$$
\end{theo}

 Theorem 10.5.1 in \cite{geometricfunctiontheory} is in fact more accurate and goes into much more details about the boundary conditions. However quoting it in a comprehensive manner would require to introduce new notations. We thus restrict ourselves to this partial result, which will satisfy our current needs.
Taking $X  = \begin{pmatrix} X_1 \\ X_2 \end{pmatrix} \in L^p \left( \D_r, \R \times \R \right)$, and $ \omega = X_1 dx + X_2dy$, one can apply theorem \ref{hodgedecompositionforme} and find a function $\alpha$, a volume form $\beta$ and a harmonic $1$-form $h$ on $\D_r$ such that: 
$$\omega  = d \alpha  + d^* \beta +h,$$ 
$$ \left\| \alpha \right\|_{W^{1,p} \left(\D_r \right) } +  \left\| \beta \right\|_{W^{1,p} \left( \D_r \right) } \le C_p\left( \D_r \right) \left\| \omega \right\|_{L^{p} \left( \D_r \right)} \le C_p\left( r \right) \left\| X \right\|_{L^{p} \left( \D_r \right)} .$$
Since $\mathrm{div}(X ) = d^* \omega = \Delta \alpha$ we deduce 

\begin{cor}
\label{hodgedecompositionmoinsbon}
Let $r>0$ and $1<p<\infty$. For any $X \in L^p \left( \D_r, \R \times \R \right)$ there exists $\alpha \in W^{1,p} \left( \D_r \right)$ such that
$$\mathrm{div}(X ) = \Delta \alpha$$
and 
$$ \left\| \alpha \right\|_{W^{1,p} \left(\D_r \right) }  \le C_p\left( r \right) \left\| X \right\|_{L^{p} \left( \D_r \right)}.$$
\end{cor}
Using Marcinkiewitz interpolation theorem (see for example theorem 3.3.3 of  \cite{bibharmmaps}) enables us to write 
\begin{cor}
\label{hodgedecompositionmoinsbon21}
Let $r>0$, for any $X \in L^{2,1} \left( \D_r, \R^2\right)$ there exists $\alpha \in W^{1,(2,1)} \left( \D_r \right)$ such that
$$\Delta \alpha=\mathrm{div} (X ) $$
and 
$$ \left\| \alpha \right\|_{W^{1,(2,1)} \left(\D_r \right) }  \le C\left( r \right) \left\| X \right\|_{L^{2,1} \left( \D_r \right)}.$$
\end{cor}

\bibliographystyle{plain}
\bibliography{bibliography}

\begin{thebibliography}{10}

\bibitem{bibrieszpot}
D.~Adams.
\newblock A note on {R}iesz potentials.
\newblock {\em Duke Math. J.}, 42(4):765--778, 1975.

\bibitem{bibnoetherwill}
Y.~Bernard.
\newblock Noether's theorem and the {W}illmore functional.
\newblock {\em Adv. Calc. Var.}, 9(3):217--234, 2016.

\bibitem{bibenergyquant}
Y.~Bernard and T.~Rivi{\`e}re.
\newblock Energy quantization for {W}illmore surfaces and applications.
\newblock {\em Ann. of Math. (2)}, 180(1):87--136, 2014.

\bibitem{MR1239099}
F.~Bethuel and J.-M. Ghidaglia.
\newblock Improved regularity of solutions to elliptic equations involving
  {J}acobians and applications.
\newblock {\em J. Math. Pures Appl. (9)}, 72(5):441--474, 1993.

\bibitem{MR0076373}
W.~Blaschke.
\newblock {\em Vorlesungen {\"u}ber {I}ntegralgeometrie}.
\newblock Deutscher Verlag der Wissenschaften, Berlin, 1955.
\newblock 3te Aufl.

\bibitem{MR1949165}
J.~Bourgain and H.~Brezis.
\newblock On the equation {${\rm div}\, Y=f$} and application to control of
  phases.
\newblock {\em J. Amer. Math. Soc.}, 16(2):393--426, 2003.

\bibitem{bubbles}
H.~Brezis and J.-M. Coron.
\newblock Convergence of solutions of {$H$}-systems or how to blow bubbles.
\newblock {\em Arch. Rational Mech. Anal.}, 89(1):21--56, 1985.

\bibitem{evalawrpde}
L.~Evans.
\newblock {\em Partial differential equations}, volume~19 of {\em Graduate
  Studies in Mathematics}.
\newblock American Mathematical Society, Providence, RI, second edition, 2010.

\bibitem{bibellipticpartialdifferentialequations}
D.~Gilbarg and N.~Trudinger.
\newblock {\em Elliptic partial differential equations of second order}.
\newblock Classics in Mathematics. Springer-Verlag, Berlin, 2001.
\newblock Reprint of the 1998 edition.

\bibitem{bibharmmaps}
F.~H{\'e}lein.
\newblock {\em Harmonic maps, conservation laws and moving frames}, volume 150
  of {\em Cambridge Tracts in Mathematics}.
\newblock Cambridge University Press, Cambridge, second edition, 2002.
\newblock Translated from the 1996 French original, With a foreword by James
  Eells.

\bibitem{geometricfunctiontheory}
T.~Iwaniec and G.~Martin.
\newblock {\em Geometric function theory and non-linear analysis}.
\newblock Oxford Mathematical Monographs. The Clarendon Press, Oxford
  University Press, New York, 2001.

\bibitem{bibkuwschat}
E.~Kuwert and R.~Sch{\"a}tzle.
\newblock The {W}illmore flow with small initial energy.
\newblock {\em J. Differential Geom.}, 57(3):409--441, 2001.

\bibitem{MR2119722}
E.~Kuwert and R.~Sch\"{a}tzle.
\newblock Removability of point singularities of {W}illmore surfaces.
\newblock {\em Ann. of Math. (2)}, 160(1):315--357, 2004.

\bibitem{bibkuwertschatzlewillmoreflow}
E.~Kuwert and Reiner Sch{\"a}tzle.
\newblock The {W}illmore flow with small initial energy.
\newblock {\em J. Differential Geom.}, 57(3):409--441, 2001.

\bibitem{MR2876249}
P.~Laurain.
\newblock Asymptotic analysis for surfaces with large constant mean curvature
  and free boundaries.
\newblock {\em Ann. Inst. H. Poincar{\'e} Anal. Non Lin{\'e}aire},
  29(1):109--129, 2012.

\bibitem{MR3843372}
P.~Laurain and T.~Rivi\`ere.
\newblock Energy quantization of {W}illmore surfaces at the boundary of the
  moduli space.
\newblock {\em Duke Math. J.}, 167(11):2073--2124, 2018.

\bibitem{biboptimalestimates}
P.~Laurain and T.~Rivi{\`e}re.
\newblock Optimal estimate for the gradient of {G}reen's function on
  degenerating surfaces and applications.
\newblock {\em Comm. Anal. Geom.}, 26(4):887--913, 2018.

\bibitem{MR3511481}
Y.~Li.
\newblock Some remarks on {W}illmore surfaces embedded in {$\Bbb{R}^3$}.
\newblock {\em J. Geom. Anal.}, 26(3):2411--2424, 2016.

\bibitem{10.1093/imrn/rnaa079}
N.~Marque.
\newblock {Minimal Bubbling for Willmore Surfaces}.
\newblock {\em International Mathematics Research Notices}, 04 2020.
\newblock rnaa079.

\bibitem{bibmulsve}
S.~M{\"u}ller and V.~{\v S}ver{\'a}k.
\newblock On surfaces of finite total curvature.
\newblock {\em J. Differential Geom.}, 42(2):229--258, 1995.

\bibitem{bibconservationlawsforconformallyinvariantproblems}
T.~Rivi{\`e}re.
\newblock Conservation laws for conformally invariant variational problems.
\newblock {\em Invent. Math.}, 168(1):1--22, 2007.

\bibitem{bibanalysisaspects}
T.~Rivi{\`e}re.
\newblock Analysis aspects of {W}illmore surfaces.
\newblock {\em Invent. Math.}, 174(1):1--45, 2008.

\bibitem{bibpcmi}
T.~Rivi{\`e}re.
\newblock Weak immersions of surfaces with {$L^2$}-bounded second fundamental
  form.
\newblock In {\em Geometric analysis}, volume~22 of {\em IAS/Park City Math.
  Ser.}, pages 303--384. Amer. Math. Soc., Providence, RI, 2016.

\bibitem{MR0268743}
H.~Wente.
\newblock An existence theorem for surfaces of constant mean curvature.
\newblock {\em Bull. Amer. Math. Soc.}, 77:200--202, 1971.

\bibitem{bibwill}
T.~J. Willmore.
\newblock {\em Riemannian geometry}.
\newblock Oxford Science Publications. The Clarendon Press, Oxford University
  Press, New York, 1993.

\end{thebibliography}

\end{document}